\documentclass[11pt,reqno]{article}
\usepackage{amsmath,amssymb,amsthm,amsfonts,amstext,amsbsy,amscd}
\usepackage{mathrsfs}
\usepackage[utf8]{inputenc}
\usepackage{bbm,dsfont}
\usepackage{xcolor}
\usepackage{comment}
\usepackage[arrow, matrix, curve]{xy}
\usepackage[colorlinks,citecolor=blue,urlcolor=blue,linkcolor = blue,hypertexnames=false]{hyperref}

\theoremstyle{plain}
\newtheorem{thm}{Theorem}

\newtheorem{corollary}{Corollary}
\newtheorem{lemma}{Lemma}
\newtheorem{assumption}{Assumption}

\theoremstyle{definition}

\theoremstyle{remark}

\DeclareMathOperator{\E}{{\mathbb E}}

\DeclareMathOperator{\R}{{\mathbb R}}

\DeclareMathOperator{\N}{{\mathbb N}}

\DeclareMathOperator{\PP}{{\mathbb P}}

\DeclareMathOperator{\M}{{\mathcal M}}
\DeclareMathOperator{\HH}{{\mathcal H}}

\numberwithin{equation}{section}

\allowdisplaybreaks[4]

\begin{document}
\title{A kernel-based analysis of Laplacian Eigenmaps}
\author{Martin Wahl\thanks{
Universit\"{a}t Bielefeld, Germany. E-mail: martin.wahl@math.uni-bielefeld.de 
\newline
\textit{2010 Mathematics Subject Classification.} 62H25, 60B20, 15A42, 47A55, 47D07, 35K08\newline
\textit{Key words and phrases.} empirical graph Laplacian, Laplace-Beltrami operator, heat semigroup, heat kernel, principal components analysis, empirical covariance operator, reproducing kernel Hilbert space, perturbation theory, concentration inequalities.}}
\date{}


\maketitle

\begin{abstract}
Given i.i.d.~observations uniformly distributed on a closed manifold $\M\subseteq \R^p$, we study the spectral properties of the associated empirical graph Laplacian based on a Gaussian kernel. Our main results are non-asymptotic error bounds, showing that the eigenvalues and eigenspaces of the empirical graph Laplacian are close to the eigenvalues and eigenspaces of the Laplace-Beltrami operator of $\M$. In our analysis, we connect the empirical graph Laplacian to kernel principal component analysis, and consider the heat kernel of $\M$ as reproducing kernel feature map. This leads to novel points of view and allows to leverage results for empirical covariance operators in infinite dimensions.
\end{abstract}

\setcounter{tocdepth}{2}

\tableofcontents

\section{Introduction}
\subsection{Motivation}
Laplacian Eigenmaps \cite{6789755} and Diffusion Maps \cite{CL} are nonlinear dimensionality reduction methods that use the eigenvalues and eigenvectors of (un)normalized graph Laplacians. Both methods are applied when the data is sampled from a low-dimensional manifold, embedded in a high-dimensional Euclidean space. From a mathematical point of view, the fundamental problem is to understand these so-called empirical graph Laplacians as spectral approximations of the underlying Laplace-Beltrami operator.

Spectral convergence properties of graph Laplacians have been studied first by \cite{Belkin2006ConvergenceOL}, based on pointwise convergence results \cite{MR2460286,GK} and perturbation theory for linear operators on Banach spaces \cite{MR2396807,MR3405533}. Since then, different approaches have been proposed in order to study the spectral properties of graph Laplacians (see e.g.~\cite{MR3299811,ZS,TGH,MR4393800,CW} and the references therein). Besides pointwise convergence results, an important role is played by Dirichlet form convergence of the empirical graph Laplacian towards the Laplace-Beltrami operator.

A different line of research studies principal component analysis (PCA) in high dimensions \cite{W,8412585} or in infinite dimensions \cite{HE15,6790375}, where the latter also includes functional and kernel PCA. Principal component analysis is a standard dimensionality reduction method that uses the eigenvalues and eigenspaces of the empirical covariance matrix or operator to find the so-called principal components, capturing most of the variance of the (empirical) distribution. Recently, there has been a lot of interest and activity in the (spectral) analysis of empirical covariance operators in infinite dimensions (see, e.g., \cite{MR2450103,MR3556768,MR4065170,RW17,MR4517351,MR4564624} and the references therein).
A related problem studies approximations of integral operators \cite{MR1781185,MR2600634}.

In this paper, we study empirical graph Laplacians through the lens of kernel PCA. In particular, we connect the eigenvalue and eigenvector problem for the empirical graph Laplacian to a kernel PCA problem, providing a novel functional analytic framework. To achieve this, the semigroup framework plays an important role. First, we interpret the Laplace-Beltrami operator of $\M$ as the generator of the heat semigroup of $\M$ \cite{Davies,G}. Second, the heat semigroup is an integral operator on the space of all square-integrable functions on $\M$, with integral kernel being the heat kernel of $\M$. Restricting its domain to the reproducing kernel Hilbert space (RKHS) induced by the heat kernel, the heat semigroup can be interpreted as the covariance operator of the random variable obtained by embedding the original observation on the manifold using the heat kernel as reproducing kernel feature map. This leads to novel points of view and allows us to use results for empirical covariance operators in infinite dimensions. In particular, this paves the way to apply Hilbert space perturbation theory in the analysis of empirical graph Laplacians.

\subsection{Laplacian Eigenmaps and Diffusion Maps}

In this section, we introduce the Laplacian Eigenmaps and the Diffusion Maps algorithm. Let $X_1,\dots,X_n$ be data points in $\R^p$. Let $\langle \cdot,\cdot\rangle$ be the standard inner product in $\R^p$ and let $\|\cdot\|_2$ be the corresponding Euclidean norm in $\R^p$. Finally, let
\begin{align}\label{eq:Guaussian:kernel}
    w_t(x,y)=\frac{1}{(4\pi t)^{d/2}}\exp\Big(-\frac{\|x-y\|_2^2}{4t}\Big),\qquad x,y\in \R^p,
\end{align}
be the Gaussian kernel with hyperparameters $d\in \N$ and $t>0$. Then the (unnormalized) empirical graph Laplacian or point cloud Laplacian with repsect to $w_t$ is given by (see, e.g.,~\cite{L})  
\begin{align*}
    \mathscr{L}_{w_t}=\frac{1}{t}(D_{w_t}-W_t)\in\R^{n\times n}
\end{align*}
with adjacency matrix $W_t\in\R^{n\times n}$ defined by
\begin{align*}
    (W_t)_{ij}=\frac{1}{n}w_t(X_i,X_j)
\end{align*}
and degree matrix $D_{w_t}\in\R^{n\times n}$ being a diagonal matrix with 
\begin{align*}
    (D_{w_t})_{ii}=\frac{1}{n}\sum_{j=1}^nw_t(X_i,X_j).
\end{align*}
Graph Laplacians have been intensively studied in spectral graph theory (see, e.g.,~\cite{C} or \cite{L} and the references therein). A first key identity of the unnormalized graph Laplacian is 
\begin{align}\label{eq:key:identity:graph:Laplacian}
    \langle u,  \mathscr{L}_{w_t}u\rangle=\frac{1}{2n}\sum_{i=1}^n\sum_{j=1}^nw_t(X_i,X_j)(u_i-u_j)^2,\qquad  u\in\R^n,
\end{align}
which can be seen as an empirical Dirichlet form (see, e.g., Proposition 1 in \cite{L}). In particular, $\mathscr{L}_{w_t}$ is symmetric and positive semi-definite. Let 
\begin{align*}
    0\leq \lambda_1(\mathscr{L}_{w_t})\leq \dots\leq \lambda_n(\mathscr{L}_{w_t})
\end{align*}
be the eigenvalues of $\mathscr{L}_{w_t}$ (in non-decreasing order), and let 
\begin{align*}
    u_1(\mathscr{L}_{w_t}),\dots,u_n(\mathscr{L}_{w_t})\in \R^n
\end{align*}
be an orthonormal basis of corresponding eigenvectors. Since $w_t(x,y)>0$ for all $x,y\in\R^n$, it follows from Proposition 2 in \cite{L} that $0=\lambda_1(\mathscr{L}_{w_t})<\lambda_2(\mathscr{L}_{w_t})$ and that $u_1(\mathscr{L}_{w_t})=n^{-1/2}(1,\dots,1)^\top$. The spectral characteristics of $\mathscr{L}_{w_t}$ are often used in spectral clustering. Indeed, given a natural number $1\leq j\leq n$, a reduced coordinate system is given by 
\begin{align}\label{eq.DMs}
    \varphi^{(j)}(X_i)=\begin{pmatrix}
        u_{1i}(\mathscr{L}_{w_t})\\ \vdots \\ u_{ji}(\mathscr{L}_{w_t})
    \end{pmatrix}\in\R^j,\qquad 1\leq i\leq n,
\end{align}
where $u_{ki}(\mathscr{L}_{w_t})$ is the $i$th coordinate of the $k$th eigenvector. Often, one also introduces a weighting by multiplying the entries $u_{ki}(\mathscr{L}_{w_t})$ by $\lambda_{k}(\mathscr{L}_{w_t})^m$ with $m>0$. Moreover, often $\mathscr{L}_{w_t}$ is replaced by the random walk graph Laplacian $I_n-D_{w_t}^{-1}\mathscr{L}_{w_t}$, where $I_n\in\R^{n\times n}$ is the identity matrix, in which case \eqref{eq.DMs} is also called (truncated) diffusion map and $\|\varphi^{(j)}(X_{i_1})-\varphi^{(j)}(X_{i_2})\|_2$ is also called diffusion distance between data points $X_{i_1}$ and $X_{i_2}$ (see, e.g., \cite{CL} and~\cite{BSS}).

\subsection{Laplace-Beltrami operator, heat semigroup and heat kernel}

Let $\mathcal{M}$ be a submanifold of $\R^p$ of dimension $d$, equipped with the Riemannian metric induced by the ambient space. We assume that $\mathcal{M}$ is closed (that is, compact and without boundary) and that the volume of $\mathcal{M}$ is $1$. Let
\begin{align}\label{eq:LB:operator}
    H=-\Delta|_{W^2(\mathcal{M})}
\end{align}
be the Laplace-Beltrami operator of $\mathcal{M}$ (see Section 3.7 in \cite{G}, Chapter 5 in \cite{Davies}, or \cite{Ros}), which we consider with domain being the Sobolev spaces $W^2(\mathcal{M})$. The sign convention implies that $H$ is a positive self-adjoint operator. Let
\begin{align}\label{eq:heat:semigroup}
    (e^{-tH})_{t> 0}
\end{align}
be the heat semigroup on $L^2(\mathcal{M})$ associated with $H$ (see, e.g., Theorem 4.9 in \cite{G} or Chapter 5 in \cite{Davies}). For each $t>0$, $e^{-tH}$ is a a positive self-adjoint bounded operator on $L^2(\mathcal{M})$. Moreover, since $\mathcal{M}$ is closed, $e^{-tH}$ is ultracontractive (see, e.g., Theorem 7.6 in \cite{G}), implying the existence of an integral kernel $k_t$ satisfying
\begin{align*}
    e^{-tH}f(x)=\int_{\mathcal{M}}k_t(x,y)f(y)dy
\end{align*}
for all $x\in\mathcal{M}$, all $t>0$ and all $f\in L^2(\mathcal{M})$, where $dy$ denotes the volume measure on $\mathcal{M}$ (see, e.g., Theorem 5.2.1 in \cite{Davies} or Theorem 7.13 in \cite{G}). The kernel $k_t$ is called heat kernel of $\M$. It is symmetric, bounded, non-negative (in fact, we have $k_t(x,y)>0$ whenever $x$ and $y$ belong to the same connected component of $\mathcal{M}$ and $k_t(x,y)=0$ otherwise), and satisfies $\int_{\M} k_t(x,y)\, dy=1$ for all $x\in\M$. In particular, $e^{-tH}$ is additionally trace-class for every $t>0$, implying that there exists an orthonormal basis $\phi_1,\phi_2,\dots$ of $L^2(\mathcal{M})$ and an non-decreasing sequence $0=\mu_1\leq \mu_2\leq \dots \rightarrow \infty$ of non-negative real numbers such that (see, e.g., Lemma 7.2.1 in \cite{Davies} or Theorem 10.13 in \cite{G})
\begin{align}\label{eq:spectral:thm}
    e^{-tH}\phi_j=\phi_j,\qquad \sum_{j=1}^\infty e^{-\mu_jt}<\infty.
\end{align}
For each $j\in \N$, $\phi_j$ is in the domain of $H$, implying that 
\begin{align*}
    H\phi_j=\lim_{t\rightarrow 0}\frac{I-e^{-tH}}{t}\phi_j=\mu_j\phi_j,
\end{align*}
where $I$ denotes the identity map. Moreover, the heat kernel admits the expansion 
\begin{align}\label{eq:Mercer:heat:kernel}
    k_t(x,y)=\sum_{j=1}^\infty e^{-\mu_jt}\phi_j(x)\phi_j(y)
\end{align}
for all $t>0$ and all $x,y\in\mathcal{M}$, which converges in $L^2(\mathcal{M}\times \mathcal{M})$ as well as absolutely and 
 uniformly in $[a,b]\times \mathcal{M}\times \mathcal{M}$ for all real numbers $b>a>0$. The multiplicity of the eigenvalue $0$ is denoted by $m$, meaning that
\begin{align*}
    0=\mu_1=\dots=\mu_m<\mu_{m+1}.
\end{align*}
The number $m$ corresponds to the number of connected components of $\mathcal{M}$. Moreover, if $\mathcal{M}_1,\dots,\mathcal{M}_m$ are the connected components of $\mathcal{M}$, then we can choose $\phi_j=\operatorname{vol}(\mathcal{M}_j)^{-1/2}\mathbf{1}_{\mathcal{M}_j}$ as the first $m$ eigenfunctions of $H$ (see, e.g., Section 2.2 in \cite{Grigor’yan_1999}).
 
\subsection{Two main consequences}

The main goal of this paper is to study the behavior of the eigenvalues and eigenvectors of $\mathscr{L}_{w_t}$ under the following manifold assumption.

\begin{assumption}\label{ass:manifold:hypothesis}
    Let $X,X_1,\dots,X_n$ be i.i.d.~random variables that are uniformly distributed on $\mathcal{M}$, where $\mathcal{M}$ is a closed submanifold of $\R^p$ with $\dim \M=d$ and $\operatorname{vol}(M)=1$ such that conditions (\hyperlink{C1}{C1}), (\hyperlink{C3}{C3}), (\hyperlink{C4}{C4}) and (\hyperlink{C5}{C5}) from below hold.
\end{assumption}

In this section, we formulate two main consequences of the analysis presented in Sections \ref{Sec:Analysis:LMAP} and \ref{sec:approx:error}. We say that an event holds with high probability, if it holds with probability at least $1-Cn^{-c}$, where $c,C>0$ are constants depending only on the manifold constants from  (\hyperlink{C1}{C1}), (\hyperlink{C3}{C3}), (\hyperlink{C4}{C4}) and (\hyperlink{C5}{C5}). The following result deals with eigenvalues.

\begin{corollary}\label{thm:thm1}
    Grant Assumption \ref{ass:manifold:hypothesis}. Then there are constants $c,C>0$ depending only on the constants in (\hyperlink{C1}{C1}), (\hyperlink{C3}{C3}), (\hyperlink{C4}{C4}) and (\hyperlink{C5}{C5}) such that the following holds. Let $n\geq 2$ and let $t\in(0,1]$ be such that $\frac{\log n}{nt^{d/2}}\leq c$. Then
    \begin{align*}
        \forall 1\leq j\leq n,\quad |\lambda_j(\mathscr{L}_{w_t})-\mu_j|\leq C\Big(\sqrt{\frac{\log n}{nt^{d/2+2}}}+\mu_jt\log^2(\tfrac{e}{t})+(\mu_j^2\vee 1)t\Big).
    \end{align*}
    with high probability.
\end{corollary}

For all $j\in\N$ let $$S_n\phi_j=\frac{1}{\sqrt{n}}(\phi_j(X_1),\dots,\phi_j(X_n))^\top\in\R^n$$ and let $O(j)$ be the set of all orthogonal $j\times j$ matrices. The next result deals with eigenspaces and provides an error bound for the reduced coordinate system in~\eqref{eq.DMs}.

\begin{corollary}\label{thm:thm2}
    Grant Assumption \ref{ass:manifold:hypothesis}. Then there are constants $c,C>0$ depending only on the constants in (\hyperlink{C1}{C1}), (\hyperlink{C3}{C3}), (\hyperlink{C4}{C4}) and (\hyperlink{C5}{C5}) such that the following holds. Let $n\geq 2$, $1\leq j\leq n$ and $t\in(0,1]$ be such that $\mu_{j+1}>\mu_j$, $\mu_{j+1}t\leq 1$ and $\frac{\log n}{nt^{d/2}}\leq c$. Then
    \begin{align*}
        \inf_{O\in O(j)}&\|(S_n\phi_1,\dots, S_n\phi_j)-(u_1(\mathscr{L}_{w_t}),\dots,u_j(\mathscr{L}_{w_t}))O\|_2\\
        &\leq C\sqrt{j}\Big(\frac{1}{\mu_{j+1}-\mu_j}\sqrt{\frac{\log n}{nt^{d/2+2}}}+\frac{\mu_{j+1}}{\mu_{j+1}-\mu_j}t\log^2(\tfrac{e}{t})+\mu_j t\Big).
    \end{align*}
    with high probability. Here, $\|\cdot\|_2$ denotes the Hilbert-Schmidt norm.
\end{corollary}

The proofs of Corollaries \ref{thm:thm1} and  \ref{thm:thm2} are based on classical (zero order) perturbation bounds (i.e.~the Weyl bound and the Davis-Kahan bound). On the one hand, this leads to non-asymptotic perturbation bounds that are valid under minimal assumptions and exhibit simple dependencies on the involved quantities ($d$, $j$, $\mu_j$, spectral gap, etc). For instance, Corollary~\ref{thm:thm1} holds without any spectral gap assumption. On the other hand, this leads to a stochastic error (i.e.~the term involving $n^{-1/2}t^{-d/4-1}$ up to a logarithmic term) that is suboptimal with respect to $t$. 

Compared to recent results in the literature, our dependence on $t$ is better than e.g.~in the bounds obtained in \cite{TGH}, but it is worse than in the bounds obtained in \cite{CW}. The latter provides a first order perturbation analysis, and establishes the stochastic errors (up to logarithmic terms) of the asymptotic order $n^{-1/2}t^{-d/4}$ and $n^{-1/2}t^{-d/4-1/2}$ for the eigenvalues and eigenfunctions, respectively. However, this requires stronger assumptions, and in \cite{CW} only a fixed number of leading eigenvalues (and eigenfunctions) are considered under the (spectral gap) assumption that they all have multiplicity~$1$. 

In a companion paper we will combine our approach (outlined in Section \ref{sec:proof:map}) with higher order perturbation expansions obtained in \cite{RW17,MR4517351,JW18b}. We conjecture that this will lead to an improved stochastic error term of the order $n^{-1/2}t^{-d/4}$ for both eigenvalues and eigenspaces (with an even better rate in the case of eigenvalues), but also leads to more restrictive assumptions on $d$, $j$, $\mu_j$, spectral gap, etc. In this paper, we formulate such an improvement only for heat kernel PCA and for the leading eigenvalue and the associated eigenspace (related to spectral clustering, see Theorem \ref{prop:improved:perturbation:bound:nullspace} below). This already requires more advanced techniques in the form of concentration inequalities (Lemmas \ref{lem:conc:eco:relative} and \ref{lem:conc:Hilbert:norm}), perturbation bounds (Lemmas \ref{lem:relative:DK} and \ref{lem:relative:Weyl}), and eigenvalues computations  (Lemma \ref{lem:eigenvalue:estimate}). Note that all these techniques are not used in the proofs of Corollaries \ref{thm:thm1} and \ref{thm:thm2}. 


\subsection{Previous and new proof map}\label{sec:proof:map}

A standard route in the analysis of the point cloud Laplacian under the manifold assumption (see, e.g.,~\cite{Belkin2006ConvergenceOL,GK,MR2238670, TGH,CW}) is described by the following error decomposition 
\begin{align*}
  \xymatrix{
      \dfrac {I-e^{-tH}}{t}\text{ on }L^2(\mathcal{M}) \quad\ar@{-}[r]^-{(1)} & \quad \mathbb{E} \mathscr{L}_{w_t}\text{ on }L^2(\mathcal{M} )\text{ or }\mathcal{C}(\mathcal{M}) \ar@{-}[d]_{(2)} \\ & \quad \mathscr{L}_{w_t}\text{ on }\mathcal{C}(\mathcal{M})
  }
\end{align*}
Here, (1) corresponds to an approximation error and (2) corresponds to a stochastic error. In the approximation error, the heat kernel $k_t$ is replaced by a Gaussian kernel $w_t$, leading to the operator
\begin{align*}
    \mathbb{E}\mathscr{L}_{w_t}f(x)=\frac{1}{t}\int_{\mathcal{M}}(f(x)-f(y))w_t(x,y)dy,\qquad f\in L^2(\mathcal{M}),x\in\mathcal{M}.
\end{align*}
The approximation error tends to decrease when $t$ decreases. It cannot be controlled uniformly over $L^2(\mathcal{M})$, but only over proper classes of smooth functions (see, e.g., Theorem 3.1 in \cite{GK}, Theorem 5.1 in \cite{CW} and \cite{MR2238670} for pointwise error bounds and Theorem 3.2 in \cite{CW} and Section 3 in \cite{TGH} for Dirichlet form error bounds). In the stochastic error, the uniform distribution on $\M$ is replaced by by the empirical distribution, leading to
\begin{align*}
    \mathscr{L}_{w_t}f(x)=\frac{1}{t}\frac{1}{n}\sum_{i=1}^n(f(x)-f(X_i))w_t(x,X_i),\qquad f\in \mathcal{C}(\mathcal{M}),x\in\mathcal{M},
\end{align*}
where $\mathcal{C}(\mathcal{M})$ denotes the space of all continuous functions on $\mathcal{M}$. The stochastic error tends to increase when $n$ increases. Again, for certain proper classes of functions, it can be controlled using concentration inequalities and empirical process theory (see, e.g., \cite{GK} and \cite{Belkin2006ConvergenceOL}).

The approach of this paper reverses the order of approximation error and stochastic error and is described by the following error decomposition
\begin{align*}
  \xymatrix{
      \dfrac {I-\Sigma_t}{t}=\dfrac{I-e^{-tH}}{t}\text{ on }\mathcal{H}_t\ar@{-}[d]_-{(1)}       \\
     \quad \dfrac{I-\hat\Sigma_t}{t}\text{ on }\mathcal{H}_t\ \quad\ar@{-}[r]^-{(2)} &  \quad\dfrac{I_n-K_t}{t}\quad  \ar@{-}[r]^-{(3)}  & \quad\mathscr{L}_{k_t}\quad \ar@{-}[r]^-{(4)} & \quad \mathscr{L}_{w_t}
  }
\end{align*}
Here, (1) corresponds to a stochastic error, (2) corresponds to the kernel trick, (3) is another stochastic error and (4) is an approximation error. The first main step will be to realize $e^{-tH}$ as the covariance operator $\Sigma_t$ of the embedded data point $k_t(X,\cdot)$ taking values in the reproducing kernel Hilbert space (RKHS) $\mathcal{H}_t$ induced by $k_t$ (Section \ref{sec:Heat:kernel:PCA}). The new stochastic error consists now in estimating $\Sigma_t$ using the empirical covariance operator $\hat\Sigma_t$ on $\HH$ based on the embedded data points $k_t(X_i,\cdot)\in \HH_t$, $1\leq i\leq n$. This error can be controlled in operator norm (which makes it easy to apply perturbation bounds for eigenvalues and eigenspaces) using concentration inequalities for random self-adjoint operators (Section \ref{sec:perturbation:Heat:kernel:PCA}). Together with the kernel trick in (2), this allows to relate the spectral characteristics of $\frac{I-e^{-tH}}{t}$ to those of  $\frac{I_n-K_t}{t}$, where $K_t\in\R^{n\times n}$ is the kernel matrix defined by $(K_t)_{ij}=n^{-1}k_t(X_i,X_j)$. Finally, the stochastic error in (3) and the approximation error in (4) can again be controlled in operator norm, using concentration inequalities and asymptotic expansions of the heat kernel (Section \ref{subsec:approx:error}).      

\paragraph{Further basic notation.}
Given a bounded (resp.~Hilbert-Schmidt) operator $A$ on a Hilbert space $\mathcal{H}$, we denote the operator norm (resp.~the Hilbert-Schmidt norm) of $A$ by $\|A\|_\infty$ (resp.~$\|A\|_2$). Given a trace class operator $A$ on $\mathcal{H}$, we denote the trace of $A$ by $\operatorname{tr}(A)$. For $g,h\in\mathcal{H}$, we denote by $g\otimes h$ the rank-one operator on $\mathcal{H}$ defined by $(g\otimes h)x=\langle h,x\rangle g$, $x\in \mathcal{H}$.  Throughout the paper, $c\in (0,1)$ and $C>1$ denote constants that may change from line to line (by a numerical value). If no further dependencies are mentioned, then these constants are absolute.

\section{Heat kernel principal component analysis}\label{Sec:Analysis:LMAP}

Throughout Section \ref{Sec:Analysis:LMAP}, we suppose that Assumption \ref{ass:manifold:hypothesis} is satisfied.

\subsection{The heat kernel as reproducing kernel}

The representation in \eqref{eq:Mercer:heat:kernel} implies that $k_t$ is a symmetric and positive semi-definite kernel function in the sense of Definition 12.6 in \cite{W}. Hence, $k_t$ induces the reproducing kernel Hilbert space (RKHS) (see, e.g., Corollary 12.26 in \cite{W})
\begin{align*}
    \HH_t=\Big\{f=\sum_{j=1}^\infty\alpha_j\phi_j\ :\ \sum_{j=1}^\infty\alpha_j^2<\infty\text{ and }\sum_{j=1}^\infty e^{\mu_jt}\alpha_j^2<\infty\Big\}
\end{align*}
with the inner product 
\begin{align*}
    \langle f,g\rangle_{\HH_t}=\sum_{j=1}^\infty e^{\mu_jt}\alpha_j\beta_j,\qquad f=\sum_{j=1}^\infty\alpha_j\phi_j, g=\sum_{j=1}^\infty\beta_j\phi_j\in\HH_t. 
\end{align*}
In particular, an orthonormal basis of $\HH_t$ is given by $$u_{t,j}=e^{-\mu_jt/2}\phi_j,\qquad j\in\N.$$
Alternatively, we have 
\begin{align*}
    \HH_t=e^{-(t/2)H}L^2(\M)
\end{align*}
with 
\begin{align*}
    \langle f,g\rangle_{\HH_t}=\langle e^{(t/2)H}f,e^{(t/2)H}g\rangle_{L^2(\M)}
    \end{align*}
for $f,g\in\HH_t$. In particular, for all $t>0$ and all $x\in\M$, we have  
\begin{align*}
    k_t(x,\cdot)\in \HH_t
\end{align*}
and the reproducing property  
\begin{align*}
 f(x)=\langle f,k_t(x,\cdot)\rangle_{\HH_t}
\end{align*}
holds. Setting $f=k_t(y,\cdot)$, we get
\begin{align}\label{eq:kernel:trick}
    k_t(x,y)=\langle k_t(y,\cdot),k_t(x,\cdot)\rangle_{\HH_t}
\end{align}
for all $t>0$ and all $x,y\in\HH_t$. Hence, geometric computations for the functions $k_t(x,\cdot)$, $x\in\M$ in the infinite-dimensional Hilbert space $\HH_t$ can be carried out using only evaluations of the kernel function.

\subsection{Heat kernel PCA}\label{sec:Heat:kernel:PCA}

Since $k_t$ is a positive semi-definite kernel function with RKHS $\HH_t$, we can consider the reproducing kernel feature map (see, e.g., \cite{6790375,MR2450103})
\begin{align*}
    \Phi_t:\M\rightarrow\HH_t,\ x\mapsto k_t(x,\cdot).
\end{align*}
This map is similar to the (population) diffusion map (see, e.g., \cite{CL}) with the difference that the range is $\HH_t$ instead of the larger space $L^2(\M)$. This feature map yields the embedded random variables 
\begin{align*}
    k_t(X,\cdot), k_t(X_1,\cdot),\dots,k_t(X_n,\cdot),
\end{align*}
taking values in the Hilbert space $\HH_t$. The random variable $k_t(X,\cdot)$ has expectation $\mathbf{1}_{\M}:\M\rightarrow\R,\ x\mapsto 1$. Moreover, by \eqref{eq:Mercer:heat:kernel}, it has the Karhunen-Loève (KL) representation 
\begin{align}\label{eq:KL:representation}
    k_t(X,\cdot)=\sum_{j=1}^\infty e^{-\mu_jt/2}\phi_j(X)u_{t,j}
\end{align}
with KL coefficients $\phi_j(X)$ satisfying $\E\phi_j(X)\phi_k(X)=0$ for all $j\neq k$ and $\E\phi_j^2(X)=1$ for all $j\in\N$. By \eqref{eq:KL:representation}, we have
\begin{align*}
    \E\|k_t(X,\cdot)\|_{\HH_t}^2=\sum_{j=1}^\infty e^{-\mu_jt}<\infty,
\end{align*}
meaning that $k_t(X,\cdot)$ is strongly square-integrable. Hence we can define the (unnormalized) covariance operator
\begin{align*}
    \Sigma_t=\E k_t(X,\cdot)\otimes k_t(X,\cdot),
\end{align*}
which is a positive self-adjoint trace class operator on $\HH_t$ (see, e.g.,  Theorem 7.2.5 in \cite{HE15} or \cite{JW18b}). The following lemma collects some basic properties of $k_t(X,\cdot)$.
\begin{lemma}\label{lem:embedded:rv}
    The random variable $k_t(X,\cdot)$ is strongly square-integrable with ovariance operator $\Sigma_t=\E k_t(X,\cdot)\otimes k_t(X,\cdot)$ satisfying
    \begin{align*}
        \Sigma_t=\sum_{j=1}^\infty e^{-\mu_jt}u_{t,j}\otimes u_{t,j}=e^{-tH}|_{\HH_t}.
    \end{align*}
    In particular, $\Sigma_t$ is a trace-class operator with
    \begin{align*}
        \operatorname{tr}(\Sigma_t)=\sum_{j=1}^\infty e^{-\mu_jt}.
    \end{align*}
\end{lemma}

\begin{proof}
    The spectral respresentation of $\Sigma_t$ is a consequence of \eqref{eq:KL:representation}. For the second equality it suffices to note that $\Sigma_tu_{t,j}=e^{-\mu_jt}u_{t,j}=e^{-tH}u_{t,j}$ for all $j\in\N$. 
\end{proof}

Lemma \ref{lem:embedded:rv} shows that the spectral characteristics of the heat semigroup can be inferred from the covariance operator of $k_t(X,\cdot)$. In particular, the eigenvalues of $\Sigma_t$ are given by $e^{-\mu_1t}\geq e^{-\mu_2t}\geq \dots >0$ (we also often write $\lambda_{t,j}=e^{-\mu_j t}$) and corresponding eigenvectors are given by $u_{t,1},u_{t,2},\dots$. For $j\in\N$, let
\begin{align*}
    P_{t,\leq j}=\sum_{k=1}^jP_{t,k},\qquad P_{t,k}=u_{t,k}\otimes u_{t,k}
\end{align*}
be the orthogonal projection onto the eigenspace spanned by $u_{t,1},\dots,u_{t,j}$. 

Let 
\begin{align*}
    \hat\Sigma_t=\frac{1}{n}\sum_{i=1}^nk_t(X_i,\cdot)\otimes k_t(X_i,\cdot)
\end{align*}
be the empirical covariance operator, which is a positive self-adjoint finite-rank operator on $\HH_t$. In particular, there exists a sequence $\hat\lambda_{t,1}\geq \hat\lambda_{t,2}\geq \cdots \geq 0$ of non-negative eigenvalues together with an orthonormal basis of eigenvectors $\hat u_{t,1},\hat u_{t,2},\dots$ in $ \HH_t$ such that 
\begin{align*}
    \hat\Sigma_t=\sum_{k=1}^\infty\hat\lambda_{t,k}\hat u_{t,k}\otimes \hat u_{t,k}.
\end{align*}
For $j\in\N$, let
\begin{align*}
    \hat P_{t,\leq j}=\sum_{k=1}^j\hat P_{t,k},\qquad \hat P_{t,k}=\hat u_{t,k}\otimes \hat u_{t,k}
\end{align*}
be the orthogonal projection onto the eigenspace spanned by $\hat u_{t,1},\dots,\hat u_{t,j}$. 

A related matrix is the kernel matrix $K_t\in\R^{n\times n}$ defined by
\begin{align*}
    (K_t)_{ij}=\frac{1}{n}\langle k_t(X_i,\cdot),k_t(X_j,\cdot)\rangle_{\HH_t}=\frac{1}{n}k_t(X_i,X_j).
\end{align*}
The so-called kernel trick states that $\hat \Sigma_t$ and $K_t$ have the same non-zero eigenvalues and that the so called principal components (see, e.g., \cite{MR3154467}) can be computed from $K_t$. To formulate this, let
\begin{align*}
    S_n:\HH_t\rightarrow \R^n,\ f\mapsto \frac{1}{\sqrt{n}}(f(X_1),\dots,f(X_n))^\top
\end{align*}
be the (normalized) sampling operator. Then there exists an orthonormal basis $\hat v_{t,1},\dots,\hat v_{t,n}$ of $\R^n$ such that (see, e.g., \cite{CW})
\begin{align*}
    K_n=\sum_{j=1}^n\hat\lambda_{t,j}\hat v_{t,j}\hat v_{t,j}^\top
\end{align*}
with principal components satisfying
\begin{align}\label{eq:PC:formula}
    \hat\lambda_{t,j}^{1/2}\hat v_{t,j}=S_n\hat u_{t,j}\in\R^n.
\end{align}

The qualitative properties in Lemma \ref{lem:embedded:rv} alone are not sufficient for a statistical analysis of heat kernel PCA. For this we need some additional quantitative properties of $k_t$. For better reference, we will formulate these properties as conditions with explicit constants.
\begin{description}
    \item[\hypertarget{C1}{(C1)}] There exists a constant $C_1>0$ such that 
    \begin{align*}
        k_t(x,x)\leq C_1t^{-d/2}
    \end{align*}
    for all $t\in(0,1]$ and all $x\in\M$.
\end{description}
Condition (\hyperlink{C1}{C1}) is satisfied under Assumption \ref{ass:manifold:hypothesis}, as follows from Theorem 7.9 in \cite{G}. It has many consequences. For instance, \eqref{eq:Mercer:heat:kernel} and (\hyperlink{C1}{C1}) yield (after integrating over $\M$)
 \begin{align*}
     je^{-\mu_jt}\leq \sum_{k=1}^\infty e^{-\mu_kt}\leq C_1t^{-d/2}
 \end{align*}
 and setting $t=\mu_j^{-1}$ with $j>m$ leads to 
 \begin{align}\label{eq:Weyls:law:lower:bound}
     \mu_j\geq c_1j^{2/d}
 \end{align}
 for all $j>m$, with a constant $c_1>0$ depending only on $C_1$ (see also Example 10.16 in \cite{G} and Remark 3 in \cite{Chavel}).
\begin{description}
    \item[\hypertarget{C2}{(C2)}] There exists a constant $C_2>0$ such that 
    \begin{align*}
       \sum_{k=1}^j\phi_k^2(x)\leq C_2j
    \end{align*}
    for all $x\in\M$ and all $j\in\N$.
\end{description}
Condition (\hyperlink{C2}{C2}) says that the orthonormal basis $\phi_1,\phi_2,\dots$ of $L^2(\M)$ is bounded on average. It holds under Assumption \ref{ass:manifold:hypothesis}, as shown in \cite{MR31145}. It is again closely related to (\hyperlink{C1}{C1}) and an upper bound for the eigenvalues. Indeed, by Weyl's formula there exists a constant $C>0$ such that $\mu_j\leq Cj^{2/d}$ for all $j\in \N$. Thus, \eqref{eq:Mercer:heat:kernel} and (\hyperlink{C1}{C1}) yield for $t=j^{-2/d}$ that 
\begin{align*}
    e^{-C}\sum_{k=1}^j\phi_k^2(x)\leq \sum_{k=1}^\infty e^{-\mu_kt}\phi_k^2(x)\leq C_1t^{-d/2}=C_1j.
\end{align*}
Finally, we state two simple consequences of (\hyperlink{C1}{C1}) and (\hyperlink{C2}{C2}).
\begin{lemma}
    If (\hyperlink{C1}{C1}) holds, then 
    \begin{align*}
        \|k_t(X,\cdot)\|_{\HH_t}^2\leq C_1t^{-d/2}\qquad \text{a.s.},
    \end{align*}
    where $C_1$ is the constant in (\hyperlink{C1}{C1}).
\end{lemma}

\begin{proof}
    By \eqref{eq:kernel:trick} we have
    \begin{align*}
        \|k_t(X,\cdot)\|_{\HH_t}^2=\langle k_t(X,\cdot),k_t(X,\cdot)\rangle_{\HH_t}=k_t(X,X),
    \end{align*}
    Hence, the claim follows from inserting (\hyperlink{C1}{C1}).
\end{proof}
Hence under (\hyperlink{C1}{C1}), $k_t(X,\cdot)$ is a bounded random variable, which is already sufficient to start the analysis of heat kernel PCA (see, e.g., \cite{MR2600634,MR4647368}). The next lemma shows that $k_t(X,\cdot)$ has a more subtle probabilistic structure if additionally (\hyperlink{C2}{C2}) is fulfilled. This will allow us to apply more refined concentration techniques.

\begin{lemma}
    If (\hyperlink{C1}{C1}) and (\hyperlink{C2}{C2}) hold, then $k_t(X,\cdot)$ has the KL representation
    \begin{align*}
    k_t(X,\cdot)=\sum_{j=1}^\infty e^{-\mu_jt/2}\phi_j(X)u_{t,j}
    \end{align*}
    such that for all $j>m$ it holds that
    \begin{align*}
        \mu_j\geq cj^{2/d} \qquad\text{and} \qquad\sum_{k=1}^j\phi_k^2(X)\leq Cj\quad \text{ almost surely.}
    \end{align*}
\end{lemma}

\begin{proof}
Follows from \eqref{eq:KL:representation}, \eqref{eq:Weyls:law:lower:bound} and (\hyperlink{C2}{C2}).
\end{proof}

\subsection{Perturbation bounds for heat kernel PCA}\label{sec:perturbation:Heat:kernel:PCA}

In this section, we derive perturbation bounds for the eigenvalues $\hat\lambda_{t,j}$ and the spectral projectors $ \hat P_{t,\leq j}$
of the empirical covariance operator $\hat\Sigma_t$, considered as spectral approximations of the eigenvalues $\lambda_{t,j}=e^{-\mu_jt}$ and the spectral projectors $ P_{t,\leq j}$ of the covariance operator $\Sigma_t$. We apply classical absolute perturbation bounds, based on the perturbation bounds by Weyl and Davis-Kahan, and we will also show how to obtain more refined bounds, by using relative perturbation bound from  \cite{RW17,MR4517351,JW18b}. In the latter case we restrict ourselves to the nullspace. 

\begin{lemma}\label{lem:conc:eco:absolute}
    Suppose that (\hyperlink{C1}{C1}) is satisfied. Then, with probability at least $1-Ct^{-d/2}e^{-y}$, $y>0$, we have
    \begin{align*}
        \|\hat\Sigma_{t}-\Sigma_t\|_{\infty}\leq C\Big(\sqrt{\frac{y}{nt^{d/2}}}+\frac{y}{nt^{d/2}}\Big)
    \end{align*}
    with constant $C>0$ depending only on $C_1$. 
\end{lemma}

\begin{proof}
    Let $\xi=k_t(X,\cdot)\otimes k_t(X,\cdot)-\Sigma_t$ and  $\xi_i=k_t(X_i,\cdot)\otimes k_t(X_i,\cdot)-\Sigma_t$ for $1\leq i\leq n$. Then we have
    \begin{align*}
        \hat\Sigma_{t}-\Sigma_t=\frac{1}{n}\sum_{i=1}^n\xi_i,
    \end{align*}
and our goal is to apply the operator Bernstein inequality stated in Appendix \ref{app:concentration:inequalities}. First,
\begin{align*}
    \|\xi\|_\infty &\leq \|k_t(X,\cdot)\otimes k_t(X,\cdot)\|_\infty+\E\|k_t(X,\cdot)\otimes k_t(X,\cdot)\|_\infty\\
    & =  \|k_t(X,\cdot)\|_{\HH_t}^2+\E \|k_t(X,\cdot)\|_{\HH_t}^2\\
    &=k_t(X,X)+\E k_t(X,X)\\
    &\leq 2C_1t^{-d/2}
\end{align*}
almost surely, where we successively applied the triangle inequality, Jensen's inequality, the reproducing property in \eqref{eq:kernel:trick}, and (\hyperlink{C1}{C1}). Similarly, we have
\begin{align*}
    \|\E\xi^2\|_\infty &\leq  \|\E(k_t(X,\cdot)\otimes k_t(X,\cdot))^2\|_\infty \\
    &=\|\E k_t(X,X)(k_t(X,\cdot)\otimes k_t(X,\cdot))\|_\infty\\
    & \leq C_1t^{-d/2}\|\Sigma_t\|_\infty\leq C_1t^{-d/2}
\end{align*}
and 
\begin{align*}
    \operatorname{tr}(\E\xi^2)&=\operatorname{tr}(\E k_t(X,X)(k_t(X,\cdot)\otimes k_t(X,\cdot)))\\
    &\leq C_1t^{-d/2}\E\operatorname{tr}(k_t(X,\cdot)\otimes k_t(X,\cdot))\\
    &=C_1^{-d/2}\E k_t(X,X)\leq C_1^2t^{-d}.
\end{align*}
An application of Lemma \ref{lem:operator:Bernstein} yields
\begin{align*}
    \PP\Big(\Big\|\frac{1}{n}\sum_{i=1}^n\xi_i\Big\|_{\infty}\geq u\Big)&\leq Ct^{-d/2}\exp\Big(-c\frac{nu^2}{t^{-d/2}+t^{-d/2}u}\Big)\\
    &\leq Ct^{-d/2}\exp(-cnt^{d/2}\min(u,u^2))
\end{align*}
for all $u>0$. Setting $$u=\max\Big(\sqrt{\frac{y}{nt^{d/2}}},\frac{y}{nt^{d/2}}\Big),$$ the claim follows.
\end{proof}

\begin{thm}\label{prop:heat:PCA:eigenvalue}
    Suppose that (\hyperlink{C1}{C1}) is satisfied. Then there are constants $c,C>0$ depending only on $C_1$ such that the following holds. Suppose that $\frac{\log n}{nt^{d/2}}\leq c$. Then, with probability at least $1-Cn^{-c}$, we have
    \begin{align*}
        \forall j\in \N,\qquad \Big|\frac{1-\hat\lambda_{t,j}}{t}-\frac{1-e^{-\mu_jt}}{t}\Big|\leq C\sqrt{\frac{\log n}{nt^{d/2+2}}}
    \end{align*}
    and
    \begin{align*}
        \forall j\in \N,\qquad \Big|\frac{1-\hat\lambda_{t,j}}{t}-\mu_j\Big|\leq C\sqrt{\frac{\log n}{nt^{d/2+2}}}+\mu_j^2t
    \end{align*}
\end{thm}

\begin{proof}
    The first claim follows from Lemma \ref{lem:conc:eco:absolute} and Lemma~\ref{lem:absolute:Weyl:DK}. The second claim follows from the inequalities $-x^2/2\leq 1-e^{-x}-x\leq 0$, $x\geq 0$. 
\end{proof}

\begin{thm}\label{prop:heat:kernel:PCA:projector:DK}
    Suppose that (\hyperlink{C1}{C1}) is satisfied. Then there are constants $c,C>0$ depending only on $C_1$ such that the following holds. Let $j\in\N$ and $t\in(0,1]$ be such that $\mu_{j+1}>\mu_j$, $\mu_{j+1}t\leq 1$ and $\frac{\log n}{nt^{d/2}}\leq c$. Then, with probability at least $1-Cn^{-c}$, we have 
    \begin{align*}
        \|\hat P_{t,\leq j}-P_{t,\leq j}\|_2\leq C\frac{\sqrt{j}}{\mu_{j+1}-\mu_j} \sqrt{\frac{\log n}{nt^{d/2+2}}}.
    \end{align*}
\end{thm}

\begin{proof}
    Combining Lemma \ref{lem:conc:eco:absolute} with the Davis-Kahan bound in Lemma \ref{lem:absolute:Weyl:DK}, we get
    \begin{align*}
        \|\hat P_{t,\leq j}-P_{t,\leq j}\|_2\leq C\frac{\sqrt{j}}{e^{-\mu_{j}t}-e^{-\mu_{j+1}t}} \sqrt{\frac{\log n}{nt^{d/2}}}
    \end{align*}
    with probability  at least $1-Cn^{-c}$. Inserting
    \begin{align*}
        \frac{1}{e^{-\mu_{j}t}-e^{-\mu_{j+1}t}}=\frac{e^{\mu_{j+1}t}}{e^{(\mu_{j+1}-\mu_j)t}-1}\leq \frac{e}{(\mu_{j+1}-\mu_j)t},
    \end{align*}
    where we also used that $\mu_{j+1}t\leq 1$ and $\mu_{j+1}>\mu_j$, the claim follows.
\end{proof}

To obtain more refined perturbation bounds for the nullspace, we introduce the relative complexity measure (see, e.g., Appendix \ref{Sec:perturbation:self-adjoint:operators})
\begin{align*}
    &\delta_{\leq m}(\hat\Sigma_t-\Sigma_t)\\
    &=\Big\|\Big(\frac{P_{t,\leq m}}{1-e^{-\mu_{m+1}t}}+R_{t,>m}\Big)^{\frac{1}{2}}(\hat\Sigma_t-\Sigma_t)\Big(\frac{P_{t,\leq m}}{1-e^{-\mu_{m+1}t}}+R_{t,>m}\Big)^{\frac{1}{2}}\Big\|_\infty
\end{align*} 
with reduced outer resolvent
\begin{align*}
    R_{t,>m}=\sum_{k>m}\frac{P_{t,k}}{1-e^{-\mu_kt}}.
\end{align*}
The following  result extends Lemma \ref{lem:conc:eco:absolute}.
\begin{lemma}\label{lem:conc:eco:relative}
    Suppose that (\hyperlink{C1}{C1}) and (\hyperlink{C2}{C2}) are satisfied. Let $d\geq 3$ and  $t\in(0,1]$ be such that $\mu_{m+1}t\leq 1$. Then, with probability at least $1-Ct^{-d/2+1}e^{-y}$, $y>0$, we have
    \begin{align*}
        \|\delta_{\leq m}(\hat\Sigma_t-\Sigma_t)\|_{\infty}\leq C\Big(\sqrt{\frac{y}{nt^{d/2+1}}}+\frac{y}{nt^{d/2}}\Big)
    \end{align*}
    with constant $C>0$ depending only on $C_1$, $C_2$ and $m$.
\end{lemma}

\begin{proof}
    For $1\leq i\leq n$, set
    \begin{align*}
        Z&=\Big(\frac{P_{t,\leq m}}{1-e^{-\mu_{m+1}t}}+R_{t,>m}\Big)^{\frac{1}{2}}X,\qquad Z_i&=\Big(\frac{P_{t,\leq m}}{1-e^{-\mu_{m+1}t}}+R_{t,>m}\Big)^{\frac{1}{2}}X_i
    \end{align*}
    and 
    \begin{align*}
        \xi=Z\otimes Z-\E(Z\otimes Z), \qquad \xi_i=Z_i\otimes Z_i-\E(Z_i\otimes Z_i).
    \end{align*}
    Then we have 
    \begin{align*}
        \|\delta_{\leq m}(\hat\Sigma_t-\Sigma_t)\|_{\infty}=\Big\|\frac{1}{n}\sum_{i=1}^n\xi_i\Big\|_{\infty}
    \end{align*}
    and our goal is to again apply the operator Bernstein inequality stated in Appendix \ref{app:concentration:inequalities}. We start with bounding $\|Z\|_{\mathcal{H}_t}$. By the KL representation in \eqref{eq:KL:representation}, we have
    \begin{align}\label{eq:norm:normalized:r:v}
        \|Z\|_{\mathcal{H}_t}^2=\frac{1}{1-e^{-\mu_{m+1}t}}\sum_{k=1}^m\phi_k^2(X)+\sum_{k>m}\frac{e^{-\mu_kt}}{1-e^{-\mu_kt}}\phi_k^2(X).
    \end{align}
    On the one hand, we have
    \begin{align}\label{eq:norm:normalized:r:v:1}
        \frac{1}{1-e^{-\mu_{m+1}t}}\sum_{k=1}^m\phi_k^2(X)&\leq \frac{e^{\mu_{m+1}t}}{e^{\mu_{m+1}t}-1}\sum_{k=1}^m\phi_k^2(X)\nonumber\\
        &\leq \frac{eC_2m}{\mu_{m+1}}t^{-1}\leq Ct^{-d/2}
    \end{align}
    with a constant $C>0$ depending only on $C_1$, $C_2$ and $m$, and where we applied $\mu_{m+1}t\leq 1$, (\hyperlink{C2}{C2}), \eqref{eq:Weyls:law:lower:bound} and $d\geq 3$. To bound the second term on the right-hand side of \eqref{eq:norm:normalized:r:v}, introduce
    \begin{align*}
        a_k=\frac{e^{-\mu_kt}}{1-e^{-\mu_kt}}\quad\text{and}\quad B_k=\sum_{l=1}^k\phi_l^2(X)
    \end{align*}
    for $k>m$. Then $a_{m+1},a_{m+2},\dots$ is a non-increasing sequence with 
    \begin{align*}
        \sum_{k>m}a_k\leq Ct^{-d/2}
    \end{align*}
     by Lemma \ref{lem:eigenvalue:estimate}, and the $B_k$ satisfy $B_k\leq C_2k$ for all $k>m$ by (\hyperlink{C2}{C2}). Hence, by partial summation, we get
     \begin{align*}
         \sum_{k=m+1}^Ma_k(B_k-B_{k-1})&=a_MB_M-a_{m+1}B_m+\sum_{k=m+1}^{M-1}(a_k-a_{k+1})B_k\\
         &\leq Ca_MM+C\sum_{k=m+1}^{M-1}(a_k-a_{k+1})k\leq C\sum_{k=m+1}^Ma_k
     \end{align*}
     for all $M>m$. Letting $M$ go to infinity, we get 
     \begin{align}\label{eq:norm:normalized:r:v:2}
         \sum_{k>m}\frac{e^{-\mu_kt}}{1-e^{-\mu_kt}}\phi_k^2(X)=\sum_{k>m}a_k(B_k-B_{k-1})\leq C\sum_{k>m}a_k\leq  Ct^{-d/2}.
     \end{align}
    Inserting \eqref{eq:norm:normalized:r:v:1} and \eqref{eq:norm:normalized:r:v:2} into \eqref{eq:norm:normalized:r:v}, we arrive at
    \begin{align*}
        \|Z\|_{\mathcal{H}_t}^2\leq Ct^{-d/2},
    \end{align*}
    with a constant $C>0$ depending only on $C_1$, $C_2$ and $m$. Next, let us bound $\|\E (Z\otimes Z)\|_\infty$. By construction, we have
    \begin{align*}
       \E (Z\otimes Z)= \Big(\frac{P_{t,\leq m}}{1-e^{-\mu_{m+1}t}}+R_{t,>m}\Big)^{1/2}\Sigma_t\Big(\frac{P_{t,\leq m}}{1-e^{-\mu_{m+1}t}}+R_{t,>m}\Big)^{1/2},
    \end{align*}
    so that 
    \begin{align*}
        \|\E (Z\otimes Z)\|_\infty=\frac{1}{1-e^{-\mu_{m+1}t}}=\frac{e^{\mu_{m+1}t}}{e^{\mu_{m+1}t}-1}\leq \frac{e}{\mu_{m+1}t}\leq \frac{C}{t},
    \end{align*}
    where we applied the spectral decomposition in Lemma \ref{lem:embedded:rv}, the condition $\mu_{m+1}t\leq 1$ and \eqref{eq:Weyls:law:lower:bound}.
    
    Equipped with the bounds for $\|Z\|_{\mathcal{H}_t}$ and $\|\E (Z\otimes Z)\|_\infty$, we can now proceed as in Lemma 2. In particular, we have
    \begin{align*}
    \|\xi\|_\infty \leq  \|Z\|_{\HH_t}^2+\E \|Z\|_{\HH_t}^2\leq Ct^{-d/2},
    \end{align*}
\begin{align}\label{eq:}
    \|\E\xi^2\|_\infty &\leq \|\E \|Z\|_{\HH_t}(Z\otimes Z)\|_\infty\leq Ct^{-d/2-1},
\end{align}
and 
\begin{align*}
    \operatorname{tr}(\E\xi^2)&=\operatorname{tr}(\|Z\|_{\HH_t}(Z\otimes Z))\leq C t^{-d/2}\E \|Z\|_{\HH_t}^2\leq C^2t^{-d}.
\end{align*}
An application of Lemma \ref{lem:operator:Bernstein} yields
\begin{align*}
    \PP\Big(\Big\|\frac{1}{n}\sum_{i=1}^n\xi_i\Big\|_{\infty}\geq u\Big)&\leq Ct^{-d/2+1}\exp\Big(-c\frac{nu^2}{t^{-d/2-1}+t^{-d/2}u}\Big)\\
    &\leq Ct^{-d/2+1}\exp(-cn\min(t^{d/2}u,u^2t^{d/2+1}))
\end{align*}
for all $u>0$. Setting $$u=\max\Big(\sqrt{\frac{y}{nt^{d/2+1}}},\frac{y}{nt^{d/2}}\Big),$$ the claim follows.
\end{proof}

\begin{thm}\label{prop:improved:perturbation:bound:nullspace}
    Suppose that (\hyperlink{C1}{C1}) and (\hyperlink{C2}{C2}) are satisfied. Let $n\geq 2$, $d\geq 5$ and $t\in(0,1]$ be such that $\mu_{m+1}t\leq 1$ and $\frac{\log n}{nt^{d/2+1}}\leq c$. Then, with probability at least $1-Cn^{-c}$, we have
    \begin{align*}
        \Big|\frac{1-\hat\lambda_{t,j}}{t}\Big|
        \leq C\sqrt{\frac{\log n}{nt^2}}+C\frac{\log n}{nt^{d/2+1}}
    \end{align*}
     for all $1\leq j\leq m$ and 
    \begin{align*}
        \|\hat P_{t,\leq m}-P_{t,\leq m}\|_2\leq C\sqrt{\frac{\log n}{nt^{d/2}}}+C\frac{\log n}{nt^{d/2+1}},
    \end{align*}
    with constants $c,C>0$ depending only on $C_1$, $C_2$ and $m$. In particular, if additionally $\frac{\log n}{nt^{d/2+2}}\leq c$ holds, then, with probability at least $1-Cn^{-c}$, we have
    \begin{align*}
        \Big|\frac{1-\hat\lambda_{t,j}}{t}\Big|
        \leq C\sqrt{\frac{\log n}{nt^{d/2}}},
    \end{align*}
     for all $1\leq j\leq m$ and 
    \begin{align*}
        \|\hat P_{t,\leq m}-P_{t,\leq m}\|_2\leq C\sqrt{\frac{\log n}{nt^{d/2}}}.
    \end{align*}
\end{thm}

Extensions to other eigenspaces can be obtained analogously, but require extensions of the perturbation bounds from \cite{W19} to the case of eigenvalues with multiplicities. 

\begin{proof}
    By Lemma \ref{lem:conc:eco:relative} we have
    \begin{align}\label{proof:improved:perturbation:bound:nullspace:1}
        \delta_{\leq m}(\hat\Sigma_t-\Sigma_t)&\leq C\Big(\sqrt{\frac{\log n}{nt^{d/2+1}}}+\frac{\log n}{nt^{d/2}}\Big)\leq 1/4
    \end{align}
    with probability at least $1-Cn^{-c}$, provided that the constant $c$ in the assumptions of Theorem \ref{prop:improved:perturbation:bound:nullspace} is chosen suffficiently small. On this event, Lemma~\ref{lem:relative:Weyl} yields 
    \begin{align*}
        \Big|\frac{1-\hat\lambda_{t,j}}{t}\Big|
        &\leq \frac{1}{t}\|P_{t,\leq m}(\hat\Sigma_t-\Sigma_t)P_{t,\leq m}\|_2+\mu_{m+1}\cdot\delta_{\leq m}(\hat\Sigma_t-\Sigma_t)^2
    \end{align*}
    for all $1\leq j\leq m$ and 
    \begin{align*}
        \|\hat P_{t,\leq m}-P_{t,\leq m}\|_2\leq \sqrt{2}\|P_{t,\leq m}(\hat\Sigma_t-\Sigma_t)R_{t,>m}\|_2+C\cdot \delta_{\leq m}(\hat\Sigma_t-\Sigma_t)^2,
    \end{align*}
    where we also used that $\mu_1=0$ and that $t^{-1}(1-e^{-\mu_{m+1}t})\leq \mu_{m+1}$ in the first inequality. We have 
    \begin{align*}
        \|P_{t,\leq m}(\hat\Sigma_t-\Sigma_t)P_{t,\leq m}\|_2^2=\sum_{j=1}^m\sum_{k=1}^m\Big(\frac{1}{n}\sum_{i=1}^n\phi_j(X_i)\phi_k(X_i)-\delta_{jk}\Big)^2
    \end{align*}
    and 
    \begin{align*}
        \|P_{t,\leq m}(\hat\Sigma_t-\Sigma_t)R_{t,>m}\|_2^2=\sum_{j=1}^m\sum_{k=m+1}^\infty \frac{e^{-\mu_k t}}{(1-e^{-\mu_kt})^2}\Big(\frac{1}{n}\sum_{i=1}^n\phi_j(X_i)\phi_k(X_i)\Big)^2,
    \end{align*}
    as can be seen from inserting the KL representation in \eqref{eq:KL:representation}. The sequence $e^{-\mu_k t}/(1-e^{-\mu_kt})^2$ is non-increasing and satisfies 
    \begin{align*}
        \sum_{k>m}\frac{e^{-\mu_k t}}{(1-e^{-\mu_kt})^2}\leq Ct^{-d/2}\quad\text{and}\quad \sum_{k>m}\frac{ke^{-\mu_k t}}{(1-e^{-\mu_kt})^2}\leq Ct^{-d},
    \end{align*}
    as follows from Lemma \ref{lem:eigenvalue:estimate}, (\hyperlink{C1}{C1}), \eqref{eq:Weyls:law:lower:bound}, and the fact that $d\geq 5$. Moreover, by (\hyperlink{C2}{C2}), we have 
    \begin{align*}
        \forall N\in\N,\qquad \Big|\sum_{j=1}^m\sum_{k=1}^N\phi_j^2(X)\phi_k^2(X)\Big|\leq C_2^2mN
    \end{align*}
    Hence, by Lemma \ref{lem:conc:Hilbert:norm}, we obtain with probability at least $1-Cn^{-c}$,
   \begin{align*}
        \sum_{j=1}^m\sum_{k=1}^m\Big(\frac{1}{n}\sum_{i=1}^n\phi_j(X_i)\phi_k(X_i)-\delta_{jk}\Big)^2\leq C\frac{\log n}{n}
    \end{align*}
    and 
    \begin{align*}
        \sum_{j=1}^m\sum_{k=m+1}^\infty \frac{e^{-\mu_k t}}{(1-e^{-\mu_kt})^2}\Big(\frac{1}{n}\sum_{i=1}^n\phi_j(X_i)\phi_k(X_i)\Big)^2&\leq C\Big(\frac{\log n}{nt^{d/2}}+\frac{\log^2 n}{n^2t^d}\Big)\\
        &\leq C\frac{\log n}{nt^{d/2}}.
    \end{align*}
    We conclude that, with probability at least $1-Cn^{-c}$,
    \begin{align*}
        \Big|\frac{1-\hat\lambda_{t,j}}{t}\Big|
        &\leq C\sqrt{\frac{\log n}{nt^2}}+C\frac{\log n}{nt^{d/2+1}}+C\frac{\log^2 n}{n^2t^d}\\
        &\leq C\sqrt{\frac{\log n}{nt^2}}+C\frac{\log n}{nt^{d/2+1}},
    \end{align*}
     for all $1\leq j\leq m$ and 
    \begin{align*}
        \|\hat P_{t,\leq m}-P_{t,\leq m}\|_2\leq C\sqrt{\frac{\log n}{nt^{d/2}}}+C\frac{\log n}{nt^{d/2+1}}.
    \end{align*}
    This completes the proof.
    \end{proof}

\section{The approximation error}\label{sec:approx:error}

\subsection{Different graph Laplacians}

Recall that $\M$ is a closed submanifold of $\R^p$ with $\dim \M=d$ and $\operatorname{vol}(M)=1$, equipped with the Riemannian metric induced by the ambient space. Let $d_{\M}$ be the intrinsic distance on $\M$ (see, e.g., Definition 4.21 in \cite{RS}), where we additionally set $d_{\M}(x,y)=+\infty$ whenever $x$ and $y$ belong to different components of $\M$. Moreover, let $\|x-y\|_2$ be the Euclidean (extrinsic) distance between $x,y\in\M$. We consider the heat kernel $k_t$, the geodesic kernel $g_t$ and the Gaussian kernel $w_t$, where the latter two are defined by
\begin{align*}
     g_t(x,y)&=\frac{1}{(4\pi t)^{d/2}}\exp\Big(-\frac{d_{\M}(x,y)^2}{4t}\Big),\\
    w_t(x,y)&=\frac{1}{(4\pi t)^{d/2}}\exp\Big(-\frac{\|x-y\|_2^2}{4t}\Big),
\end{align*}
for $x,y\in \M$. Based on these kernels, we introduce the following un-normalized graph Laplacians
\begin{align*}
    \mathscr{L}_{k_t}&=\frac{ D_{k_t}-K_t}{t}\in\R^{n\times n},\\
    \mathscr{L}_{g_t}&=\frac{ D_{g_t}-G_t}{t}\in\R^{n\times n},\\
    \mathscr{L}_{w_t}&=\frac{D_{w_t}-W_t}{t}\in\R^{n\times n},
\end{align*}
where
\begin{align*}
    (K_t)_{ij}&=\frac{1}{n}k_t(X_i,X_j),\\
    (G_t)_{ij}&=\frac{1}{n}g_t(X_i,X_j),\\
    (W_t)_{ij}&=\frac{1}{n}w_t(X_i,X_j),
\end{align*}
and the degree matrices $D_{k_t}$, $D_{g_t}$ and $D_{w_t}$ are diagonal matrices with 
\begin{align*}
    (D_{k_t})_{ii}&=\frac{1}{n}\sum_{j=1}^n k_t(X_i,X_j),\\
    (D_{g_t})_{ii}&=\frac{1}{n}\sum_{j=1}^n g_t(X_i,X_j),\\
    (D_{w_t})_{ii}&=\frac{1}{n}\sum_{j=1}^n w_t(X_i,X_j).
\end{align*}

\subsection{Analysis of the approximation error}\label{subsec:approx:error}

In order to relate the different graph Laplacians, we need relationships between the extrinsic and the intrinsic distance and between the geodesic kernel and the heat kernel. Again, for better reference, we formulate these relationships as conditions with explicit constants and refer to the literature where these properties are proven. 

\begin{description}
    \item[\hypertarget{C3}{(C3)}] (Extrinsic and intrinsic distance). Suppose that there are constants $c_3\in(0,1)$ and $C_3>1$ such that 
    \begin{align*}
        0\leq d_{\M}^2(x,y)-\|x-y\|_2^2\leq C_3d_{\M}^4(x,y)
    \end{align*}
    for all $x,y\in\M$, $d_{\M}(x,y)<c_3$ and
    \begin{align*}
        \inf_{x,y\in\M, d_{\M}(x,y)>s}\|x-y\|_2^2\geq c_3d_{\M}^2(x,y)
    \end{align*}
    for all $0<s<c_3$.
\end{description}
The first condition is standard in the manifold learning literature (see, e.g., \cite{GK,Belkin2006ConvergenceOL,CW}), and proven by the exponential map and Taylor expansions. The second condition measures the bottleneckness of $\M$ and the distance between the different connected components of $\M$. 

\begin{lemma}\label{lem:approx:g:w}
    Suppose that (\hyperlink{C3}{C3}) is satisfied. Let $K\in\N$. Then there exist constants $C>0$ and $c\in(0,1)$ depending only on $c_3,C_3$, $d$ and $K$ such that 
    \begin{align*}
        |w_t(x,y)-g_t(x,y)|\leq Cg_t(x,y)t\log^2(\tfrac{e}{t})+Ct^K
    \end{align*}
    for all $x,y\in\M$ and all $t\in(0,c)$.
\end{lemma}

\begin{proof}
    Let $t\in(0,1]$ be such that 
    \begin{align*}
        \delta(t):=\sqrt{\frac{4(K+d/2)}{c_3}t\log(\tfrac{e}{t})}<c_3.
    \end{align*}
    Let $x,y\in\M$ be arbitrary. If $d_{\M}(x,y)>\delta(t)$, then (\hyperlink{C3}{C3}) yields $\|x-y\|_2^2>c_3\delta(t)^2$ and thus
    \begin{align}
        0\leq w_t(x,y)-g_t(x,y)&\leq w_t(x,y)\leq \frac{1}{(4\pi t)^{d/2}}\exp\Big(-c_3\frac{\delta(t)^2}{4t}\Big)\nonumber\\
        & \leq  Ct^{-d/2}t^{K+d/2}=Ct^{K}\label{eq:g:w:global}.
    \end{align}
    On the other hand, if $d_{\M}(x,y)\leq \delta(t)$, then (\hyperlink{C3}{C3}) yields
    \begin{align}
        0\leq w_t(x,y)-g_t(x,y)&\leq g_t(x,y)\exp\Big(C_3\frac{d_{\M}^4(x,y)}{4t}\Big)-g_t(x,y)\nonumber\\
        &\leq g_t(x,y)(\exp(Ct\log^2(\tfrac{e}{t}))-1)\nonumber\\
        &\leq Cg_t(x,y)t\log^2(\tfrac{e}{t}).\label{eq:g:w:local}
    \end{align}
    The claim follows from combining \eqref{eq:g:w:global} and \eqref{eq:g:w:local}.
\end{proof}

\begin{description}
    \item[\hypertarget{C4}{(C4)}] (First order asymptotic expansion of the heat kernel). There are smooth functions $u$ and $r_t$ on $\M\times \M$ and constants $c_4\in(0,1)$ and $C_4>0$ such that 
    \begin{align*}
        k_t(x,y)=g_t(x,y)(1+u(x,y)+r_t(x,y))
    \end{align*}
    with
    \begin{align*}
        |u(x,y)|\leq C_4d_{\M}^2(x,y)\quad\text{and}\quad|r_t(x,y)|\leq C_4t
    \end{align*}
    for all $x,y\in\M$ with $d_{\M}(x,y)<c_4$ and all $t\in(0,1]$.
\end{description}
If $x$ and $y$ belong to different components, then $k_t(x,y)=g_t(x,y)=0$ and the above condition is trivially satisfied with $u(x,y)=r_t(x,y)=0$. If $x$ and $y$ belong to the same connected component, then (\hyperlink{C4}{C4}) is a consequence of the parametrix construction. See, e.g., Equation (45) in \cite{Chavel} (which is an improvement over Exercise 5 in \cite{Ros}) and also \cite{GK} for an upper bound on $u$.

\begin{description}
    \item[\hypertarget{C5}{(C5)}] (Refined global heat kernel). There are constants $c_5\in(0,1/4]$ and $C_5>0$ such that  
    \begin{align*}
        k_t(x,y)\leq \frac{C_5}{t^{d/2}}\exp\Big(-c_5\frac{d^2_{\M}(x,y)}{t}\Big)
    \end{align*}
    for all $x,y\in\M$ and all $t\in(0,1]$.
\end{description}

\begin{lemma}\label{lem:approx:g:k}
    Suppose that (\hyperlink{C4}{C4}) and (\hyperlink{C5}{C5}) are satisfied. Let $K\in\N$. Then there exist  constants $C>0$ and $c\in(0,1)$ depending only on $c_4,c_5,C_4,C_5$, $d$ and $K$ such that 
    \begin{align*}
        |k_t(x,y)-g_t(x,y)|\leq Ck_t(x,y)t\log(\tfrac{e}{t})+Ct^K
    \end{align*}
    for all $x,y\in\M$ and all $t\in(0,c)$.
\end{lemma}

\begin{proof}
    Let $t\in(0,1]$ be such that 
    \begin{align*}
        \delta(t):=\sqrt{\frac{(K+d/2)}{c_5}t\log(\tfrac{e}{t})}<c_4.
    \end{align*}
    Let $x,y\in\M$ be arbitrary. If $d_{\M}(x,y)>\delta(t)$, then (\hyperlink{C5}{C5}) yields
    \begin{align}
        k_t(x,y)\leq C_5t^{-d/2}e^{-c_5\delta^2(t)/t}\leq Ct^{-d/2}t^{K+d/2}=Ct^K\label{eq:g:k:global:1}
    \end{align}
    and, since $c_5<1/4$,
    \begin{align}\label{eq:g:k:global:2}
        g_t(x,y)\leq (4\pi t)^{-d/2}e^{-\delta^2(t)/(4t)}\leq Ct^{-d/2}t^{K+d/2}=Ct^K.
    \end{align}
    On the other hand, if $d_{\M}(x,y)\leq \delta(t)$, then (\hyperlink{C4}{C4}) yields
    \begin{align*}
        k_t(x,y)&=(1+u(x,y)+r_t(x,y))\\
        &\leq g_t(x,y)(1+C_4\delta^2(t)+C_4t)\\
        &\leq g_t(x,y)(1+Ct\log(\tfrac{e}{t})) 
    \end{align*}
    and similarly 
    \begin{align*}
        k_t(x,y)&\geq g_t(x,y)(1-Ct\log(\tfrac{e}{t})).
    \end{align*}
    Hence, if additionally $Ct\log(\tfrac{e}{t})\leq 1/2$ holds, then 
    \begin{align*}
        \frac{1}{1+Ct\log(\tfrac{e}{t})}k_t(x,y)\leq g_t(x,y)\leq \frac{1}{1-Ct\log(\tfrac{e}{t})}k_t(x,y)
    \end{align*}
    and, by elementary inequalities,
     \begin{align}\label{eq:g:k:local}
        (1-Ct\log(\tfrac{e}{t}))k_t(x,y)\leq g_t(x,y)\leq (1+2Ct\log(\tfrac{e}{t}))k_t(x,y).
    \end{align}
    The claim follows from \eqref{eq:g:k:global:1}, \eqref{eq:g:k:global:2} and \eqref{eq:g:k:local}.
\end{proof}

\begin{corollary}\label{lem:approx:k:w} 
    Suppose that (\hyperlink{C3}{C3}), (\hyperlink{C4}{C4}) and (\hyperlink{C5}{C5}) are satisfied. Let $K\in\N$. Then there exist constants $C>0$ and $c\in(0,1)$ depending only on $c_3,c_4,c_5,C_3,C_4,C_5$, $d$ and $K$ such that 
    \begin{align*}
        |k_t(x,y)-w_t(x,y)|\leq Ck_t(x,y)t\log^2(\tfrac{e}{t})+Ct^K
    \end{align*}
    for all $x,y\in\M$ and all $t\in(0,c)$.
\end{corollary}

\begin{proof}
    Follows from Lemmas \ref{lem:approx:g:w} and \ref{lem:approx:g:k}.
\end{proof}

We now state the main result of this section. 

\begin{thm}\label{prop:approx:error}
    Suppose that (\hyperlink{C3}{C3}), (\hyperlink{C4}{C4}) and (\hyperlink{C5}{C5}) are satisfied. Let $K\in\N$. Then there exist constants $C>0$ and $c\in(0,1)$ depending only on $c_3,c_4,c_5,C_3,C_4,C_5$, $d$ and $K$ such that 
    \begin{align*}
        |\langle u,\mathscr{L}_{k_t}u\rangle-\langle u,\mathscr{L}_{w_t}u\rangle|\leq Ct\log^2(\tfrac{e}{t})(\langle u,\mathscr{L}_{k_t}u\rangle+t^{K}\|u\|_2^2)
    \end{align*}
    for all $ u\in\R^n$ and all $t\in (0,c)$.
\end{thm}

\begin{proof}
    Let $c,C>0$ be constants such that the conclusion of Corollary \ref{lem:approx:k:w} holds with $K+2$. Let $t\in(0,c)$ and let $u\in\R^n$. Then
    \begin{align*}
        &|\langle u, (D_{k_t}-K_t)u\rangle-\langle u, (D_{w_t}-W_t)u\rangle|\\
        &=\Big|\frac{1}{2n}\sum_{i=1}^n\sum_{j=1}^nk_t(X_i,X_j)(u_i-u_j)^2-\frac{1}{2n}\sum_{i=1}^n\sum_{j=1}^nw_t(X_i,X_j)(u_i-u_j)^2\Big|\\
        &\leq \frac{1}{2n}\sum_{i=1}^n\sum_{j=1}^n|k_t(X_i,X_j)-w_t(X_i,X_j)|(u_i-u_j)^2\\
        &\leq \frac{1}{2n}\sum_{i=1}^n\sum_{j=1}^nCk_t(X_i,X_j)t\log^2(\tfrac{e}{t})(u_i-u_j)^2+\frac{1}{2n}\sum_{i=1}^n\sum_{j=1}^nCt^{K+2}(u_i-u_j)^2,
    \end{align*}
    where we applied \eqref{eq:key:identity:graph:Laplacian} (to both kernels $w_t$ and $k_t$) and Corollary \ref{lem:approx:k:w}. Hence, we arrive at 
       \begin{align*}
        &|\langle u, (D_{k_t}-K_t)u\rangle-\langle u, (D_{w_t}-W_t)u\rangle|\\
        &\leq Ct\log^2(\tfrac{e}{t})\langle u, (D_{k_t}-K_t)u\rangle+C\frac{t^{K+2}}{n}\sum_{i=1}^n\sum_{j=1}^n(u_i^2+u_j^2)\\
        &=Ct\log^2(\tfrac{e}{t})\langle u, (D_{k_t}-K_t)u\rangle+2Ct^{K+2}\|u\|_2^2.
    \end{align*}
    Dividing through $t$, the claim follows. 
\end{proof}

\begin{corollary}\label{corollary:relative:perturbation:graph:Laplacian}
    Under the assumptions of Theorem \ref{prop:approx:error}, we have
    \begin{align*}
        \|(\mathscr{L}_{k_t}+t^KI_n)^{-1/2}(\mathscr{L}_{w_t}-\mathscr{L}_{k_t})(\mathscr{L}_{k_t}+t^KI_n)^{-1/2}\|_\infty\leq Ct\log^2(\tfrac{e}{t})
    \end{align*}
    for all $t\in(0,c)$.
\end{corollary}

\section{Additional proofs}
\subsection{Proof of Corollary \ref{thm:thm1}}

We proceed as in the proof map given in Section \ref{sec:proof:map}. The claim follows from the following three steps and the triangle inequality. Recall that we say that an event holds with high probability, if it holds with probability at least $1-Cn^{-c}$, where $c,C>0$ are constants depending only on the constants from (\hyperlink{C1}{C1}), (\hyperlink{C3}{C3}), (\hyperlink{C4}{C4}) and (\hyperlink{C5}{C5}). In the following, we often make use of the fact that if two events hold with high probability, their union also holds with high probability.

\paragraph{Reduction to heat kernel PCA.} We have
\begin{align*}
    \forall 1\leq j\leq n,\qquad \Big|\lambda_j\Big(\frac{I-K_n}{t}\Big)-\mu_j\Big|\leq C\Big(\sqrt{\frac{\log n}{n t^{d/2+2}}}+\mu_j^2t\Big)
\end{align*}
with high probability.

\begin{proof}
    By Theorem \ref{prop:heat:PCA:eigenvalue}, we have
    \begin{align*}
        \forall j\in\N,\qquad \Big|\frac{1-\hat\lambda_{t,j}}{t}-\mu_j\Big|\leq C\sqrt{\frac{\log n}{n t^{d/2+2}}}+\frac{\mu_j^2t}{2}
    \end{align*}
    with high probability. Hence, the claim follow from the fact that the first $n$ eigenvalues of $\hat\Sigma_t$ coincide with those of $K_t$.
\end{proof}

\paragraph{Reduction to the heat graph Laplacian.}  We have
\begin{align}\label{eq:reduction:heat:graph:Laplacian}
    \forall 1\leq j\leq n,\qquad \Big|\lambda_j\Big(\frac{I_n-K_t}{t}\Big)-\lambda_j(\mathscr{L}_{k_t})\Big|\leq C\sqrt{\frac{\log n}{n t^{d/2+2}}}
\end{align}
with high probability.
\begin{proof}
    By Weyl's bound, we have
    \begin{align}\label{eq:Weyl:diagonal:perturbation}
        \Big|\lambda_j\Big(\frac{I_n-K_t}{t}\Big)-\lambda_j(\mathscr{L}_{k_t})\Big|\leq  \frac{1}{t}\|I_n-D_{k_t}\|_{\infty},
    \end{align}
    with diagonal matrix $I_n-D_{k_t}$ such that          
    \begin{align*}
        \|I_n-D_{k_t}\|_{\infty}=\max_{1\leq i\leq n}\Big|\frac{1}{n}\sum_{j=1}^nk_t(X_i,X_j)-1\Big|.
    \end{align*}
    Fix $1\leq i\leq n$. Then
    \begin{align*}
        \Big|\frac{1}{n}\sum_{j=1}^n(k_t(X_i,X_j)-1)\Big|\leq \frac{|k_t(X_i,X_i)-1|}{n}+\Big|\frac{1}{n}\sum_{j=1, j\neq i}^n(k_t(X_i,X_j)-1)\Big|.
    \end{align*}
    From (\hyperlink{C1}{C1}) it follows that
    \begin{align*}
        k_t(X_i,X_i)\leq C_1t^{-d/2}
    \end{align*}
    almost surely and thus
    \begin{align*}
        \frac{|k_t(X_i,X_i)-1|}{n}\leq \frac{C_1t^{-d/2}+1}{n}\leq C\frac{1}{nt^{d/2}}.
    \end{align*}
    Moreover for all $j\neq i$, we have
    \begin{align*}
        k_t(X_i,X_j)\leq \sqrt{k_t(X_i,X_i)\cdot k_t(X_j,X_j)}\leq C_1t^{-d/2}
    \end{align*}
    almost surely and 
    \begin{align*}
        \E_{X_j}(k_t(X_i,X_j)-1)^2\leq \E_{X_j}k_t(X_i,X_j)^2\leq C_1t^{-d/2}\E_{X_j}k_t(X_i,X_j)=C_1t^{-d/2},
    \end{align*}
    where $\E_{X_j}$ denotes expectation with respect to $X_j$ only. Hence, an application of Bernstein's inequality (conditional on $X_i$) yields
    \begin{align*}
        \PP\Big(\Big|\frac{1}{n}\sum_{j=1, j\neq i}^n(k_t(X_i,X_j)-1)\Big|\geq u\Big)\leq 2\exp(cnt^{d/2}\min(u,u^2))
    \end{align*}
    for all $u>0$. Setting 
    $$u=C\cdot\max\Big(\sqrt{\frac{\log n}{nt^{d/2}}},\frac{\log n}{nt^{d/2}}\Big)\leq C\sqrt{\frac{\log n}{nt^{d/2}}},$$
    with $C$ sufficiently large, we obtain
    \begin{align}\label{eq:conc:diagonal:perturbation}
         \|I_n-D_{k_t}\|_{\infty}\leq C\sqrt{\frac{\log n}{nt^{d/2}}}
    \end{align}
    with high probability, where we also applied the union bound. Inserting this into \eqref{eq:Weyl:diagonal:perturbation}, the claim follows.
\end{proof}

\paragraph{Reduction to the Gaussian graph Laplacian.}  We have
\begin{align*}
    \forall 1\leq j\leq n,\qquad |\lambda_j(\mathscr{L}_{k_t})-\lambda_j(\mathscr{L}_{w_t})|\leq C\Big(\sqrt{\frac{\log n}{n t^{d/2+2}}}+(\mu_j^2\vee 1)t\log^2(\tfrac{e}{t})\Big).
\end{align*}
with high probability.

\begin{proof}
For $K\in \N$ let
\begin{align*}
    \tilde{\mathscr{L}}_{k_t}=\mathscr{L}_{k_t}+t^KI_n,\qquad \tilde{\mathscr{L}}_{w_t}=\mathscr{L}_{w_t}+t^KI_n.
\end{align*}
Then $\tilde{\mathscr{L}}_{k_t}$ and $\tilde{\mathscr{L}}_{w_t}$ are strictly positive and Corollary \ref{corollary:relative:perturbation:graph:Laplacian} yields
\begin{align*}
        \|\tilde{\mathscr{L}}_{k_t}^{-1/2}(\tilde{\mathscr{L}}_{w_t}-\tilde{\mathscr{L}}_{k_t})\tilde{\mathscr{L}}_{k_t}^{-1/2}\|_\infty\leq Ct\log^2(\tfrac{e}{t})
    \end{align*}
for all $t\in(0,c)$. An application of Lemma \ref{lemma:graph:Laplace:version:Weyl:bound} yields 
\begin{align*}
    |\lambda_j(\mathscr{L}_{k_t})-\lambda_j(\mathscr{L}_{w_t})|&=|\lambda_j(\tilde{\mathscr{L}}_{k_t})-\lambda_j(\tilde{\mathscr{L}}_{w_t})|\leq  C\lambda_j(\tilde{\mathscr{L}}_{k_t}) t\log^2(\tfrac{e}{t})
\end{align*}
for all $t\in(0,c)$ and all $1\leq j\leq n$. By the previous two steps this implies that 
\begin{align*}
    |\lambda_j(\mathscr{L}_{k_t})-\lambda_j(\mathscr{L}_{w_t})|&\leq C\Big(\mu_j+\mu_j^2t+t^K+\sqrt{\frac{\log n}{n t^{d/2+2}}}\Big)t\log^2(\tfrac{e}{t})\\
    &\leq C\Big(\mu_jt\log^2(\tfrac{e}{t})+(\mu_j^2\vee 1)t+\sqrt{\frac{\log n}{n t^{d/2+2}}}\Big)
\end{align*}
for all $t\in(0,c)$ and all $1\leq j\leq n$.
\end{proof}

\subsection{Proof of Corollary \ref{thm:thm2}}

We proceed as in the proof map given in Section \ref{sec:proof:map}. In particular, the claim follows from the following steps and the triangle inequality.

\paragraph{Preliminaries.} Recall that we write $\lambda_{t,k}=e^{-\mu_kt}$ for all $k\in\N$. For all $k\in \N$, we have 
\begin{align*}
    \|S_n\phi_k\|_2^2=\langle \phi_k,\hat\Sigma_{t}\phi_k\rangle_{\HH_t}&\leq \hat\lambda_{t,k}\cdot\|\phi_k\|_{\HH_t}^2\\
    &\leq (1+\|\hat\Sigma_t-\Sigma_t\|_{\infty})\cdot e^{\mu_k t},
\end{align*}
where we used Weyl's bound and the fact that $\lambda_{t,k}\leq 1$ for all $k\in\N$ in the last inequality. Hence, by Lemma \ref{lem:conc:eco:absolute} and the assumptions in Corollary \ref{thm:thm2}, we have for all $1\leq k\leq j$,
\begin{align*}
    \|S_n\phi_k\|_2^2\leq \Big(1+C\sqrt{\frac{\log n}{nt^{d/2}}}\Big)e^{\mu_kt}\leq Ce^{\mu_kt}\leq Ce
\end{align*}
with high probability. Consequently, we may assume in the following that
\begin{align}\label{ass:Theorem:2:proof}
    \sqrt{\frac{\log n}{nt^{d/2+2}}}+\mu_{j+1}t\log^2(\tfrac{e}{t})\leq c(\mu_{j+1}-\mu_j),
\end{align}
for some sufficiently small constant $c>0$ (to be determined in the following steps), since otherwise the claim follows from the fact that the inequality
\begin{align*}
    \inf_{O\in O(j)}&\|(S_n\phi_1,\dots, S_n\phi_j)-(u_1(\mathscr{L}_{w_t}),\dots,u_j(\mathscr{L}_{w_t}))O\|_2\\
    &\leq \|(S_n\phi_1,\dots, S_n\phi_j)\|_2+\|(u_1(\mathscr{L}_{w_t}),\dots,u_j(\mathscr{L}_{w_t}))\|_2\\
    &=\Big(\sum_{k=1}^j\|S_n\phi_k\|_2^2\Big)^{1/2}+\sqrt{j}\leq C\sqrt{j}
\end{align*}
holds with high probability. Let us also collect some perliminary eigenvalue bounds. First, with high probability, we have
\begin{align}
    \hat\lambda_{t,j}\geq \lambda_{t,j}-\|\hat\Sigma_t-\Sigma_t\|_{\infty}&\geq e^{-\mu_jt}-C\sqrt{\frac{\log n}{nt^{d/2}}}\nonumber\\
    & \geq e^{-1}-C\sqrt{\frac{\log n}{nt^{d/2}}}\geq \frac{1}{2e},\label{eq:eigenvalue:positive}
\end{align}
provided that $c>0$ from the assumptions is chosen small enough (such that $c\leq 1/(2Ce)$), meaning that $\hat\lambda_{t,1}\geq \cdots\geq \hat\lambda_{t,j}>0$ with high probability. Moreover, with high probability, we have
\begin{align*}
    |1-\hat\lambda_{t,k}|\leq 1-e^{-\mu_k t}+\|\hat\Sigma_t-\Sigma_t\|_{\infty}\leq 1-e^{-\mu_k t}+C\sqrt{\frac{\log n}{nt^{d/2}}}\leq c<1
\end{align*}
for all $1\leq k\leq j$.
Hence, with high probability, we have
\begin{align}
    |\lambda_{t,k}^{-1/2}-1|&\leq C|1-\lambda_{t,k}|\leq Ct\mu_k\nonumber\\
    |\hat \lambda_{t,k}^{-1/2}-1|&\leq C|1-\hat\lambda_{t,k}|\leq C\Big(t\mu_k+\sqrt{\frac{\log n}{nt^{d/2}}}\Big)\label{eq:prelimiary:proof:Thm:2}
\end{align}
for all $1\leq k\leq j$, where we also used that $|(1-x)^{-1/2}-1|\leq C|x|$ for all $|x|\leq c<1$.

\paragraph{Reduction to heat kernel PCA.}  We have
\begin{align*}
        \inf_{O\in O(j)}&\|(S_n\phi_1,\dots, S_n\phi_j)-(\hat v_{t,1},\dots,\hat v_{t,j})O\|_2\\
        &\leq C\Big(\frac{\sqrt{j}}{\mu_{j+1}-\mu_j}\sqrt{\frac{\log n}{nt^{d/2+2}}}+\sqrt{j}\mu_{j}t\Big)
    \end{align*}
    with high probability.
\begin{proof}
    If $h_1,\dots,h_j$ are elements in $\HH_t$, then we consider $(h_k)_{k=1}^j=(h_1,\dots,h_j)$ as a linear map from $\R^j$ to $\HH_t$, mapping $\alpha\in\R^j$ to $\sum_{k=1}^j\alpha_kh_k$. Now, on the one hand, we have $\phi_k=\lambda_{t,k}^{-1/2}u_{t,k}$ for all $k\in\N$. On the other hand, \eqref{eq:eigenvalue:positive} holds with high probability, implying that $\hat v_{t,k}=\hat\lambda_{t,k}^{-1/2}S_n \hat u_{t,k}$ for all $1\leq k\leq j$ by \eqref{eq:PC:formula}. We estimate with high probability
\begin{align*}
    \inf_{O\in O(j)}&\|(S_n\phi_1,\dots, S_n\phi_j)-(\hat v_{t,1},\dots,\hat v_{t,j})O\|_2\\
    &=\inf_{O\in O(j)}\|S_n(\lambda_{t,k}^{-1/2}u_{t,k})_{k=1}^j-S_n(\hat \lambda_{t,k}^{-1/2}\hat u_{t,k})_{k=1}^jO\|_2\\
    &\leq \|S_n\|_\infty\cdot\inf_{O\in O(j)}\|(\lambda_{t,k}^{-1/2}u_{t,k})_{k=1}^j-(\hat \lambda_{t,k}^{-1/2}\hat u_{t,k})_{k=1}^jO\|_2\\
    &\leq C\cdot\inf_{O\in O(j)}\|(u_{t,k})_{k=1}^j-(\hat u_{t,k})_{k=1}^jO\|_2\\
    &+C\|((\lambda_{t,k}^{-1/2}-1)u_{t,k})_{k=1}^j\|_2+C\|((\hat\lambda_{t,k}^{-1/2}-1)\hat u_{t,k})_{k=1}^j\|_2\\
    &\leq C\|P_{t,\leq j}-\hat P_{t,\leq j}\|_2\\
    &+C\sqrt{j}\max_{1\leq k\leq j}|\lambda_{t,k}^{-1/2}-1|+C\sqrt{j}\max_{1\leq k\leq j}|\hat \lambda_{t,k}^{-1/2}-1|,
\end{align*}
where we applied the fact that $\|S_n\|_\infty^2=\hat\lambda_{t,1}\leq C$ with high probability in the second inequality and Lemma \ref{lemma:char:subspace:distance} in the last inequality. Inserting Theorem \ref{prop:heat:kernel:PCA:projector:DK},  \eqref{eq:prelimiary:proof:Thm:2}, \eqref{ass:Theorem:2:proof} and $\mu_{j+1}t\leq 1$ the claim follows.
\end{proof}

\paragraph{Reduction to the heat graph Laplacian.}  We have
\begin{align*}
        &\inf_{O\in O(j)}\|(\hat v_{t,1},\dots,\hat v_{t,j})-(u_1(\mathscr{L}_{k_t}),\dots,u_j(\mathscr{L}_{k_t}))O\|_2\\
        &\leq \|P_{\leq j}(\tfrac{1}{t}(I_n-K_t))-P_{\leq j}(\mathscr{L}_{k_t})\|_2\leq C\frac{1}{\mu_{j+1}-\mu_j}\sqrt{\frac{\log n}{nt^{d/2+2}}}.
    \end{align*}
    with high probability.
\begin{proof}
    The first equality holds by Lemma \ref{lemma:char:subspace:distance}. To see the second claim, first note that
    \begin{align*}
        \frac{1-\lambda_{t,j+1}}{t}-\frac{1-\lambda_{t,j}}{t}&=\frac{e^{-\mu_j t}-e^{-\mu_{j+1}t}}{t}\\
        &=\frac{e^{(\mu_{j+1}-\mu_j)t}-1}{te^{\mu_{j+1}t}} \geq \frac{(\mu_{j+1}-\mu_j)t}{te}=\frac{\mu_{j+1}-\mu_j}{e},
    \end{align*}
    where we applied the assumption $\mu_{j+1}t\leq 1$ in the inequality. Combining this with Weyl's bound and Lemma \ref{lem:conc:eco:absolute}, we arrive at 
    \begin{align*}
        &\lambda_{j+1}\Big(\frac{I_n-K_t}{t}\Big)-\lambda_{j}\Big(\frac{I_n-K_t}{t}\Big)=\frac{1-\hat\lambda_{t,j+1}}{t}-\frac{1-\hat\lambda_{t,j}}{t}\\
        &\geq \frac{1-\lambda_{t,j+1}}{t}-\frac{1-\lambda_{t,j}}{t}-\frac{2}{t}\|\hat\Sigma_t-\Sigma_t\|_{\infty}\\
        &\geq \frac{\mu_{j+1}-\mu_j}{e}-C\sqrt{\frac{\log n}{nt^{d/2+2}}}\\
        &\geq \frac{\mu_{j+1}-\mu_j}{2e},
    \end{align*}
    with high probability, provided that $c>0$ in \eqref{ass:Theorem:2:proof} is chosen small enough (such that $c\leq 1/(2Ce)$). Combining this with the Davis-Kahan bound in Corollary \ref{corollary:DK}, we get
    \begin{align}\label{eq:DK:diagonal:perturbation}
        \|P_{\leq j}(\tfrac{1}{t}(I_n-K_t))-P_{\leq j}(\mathscr{L}_{k_t})\|_2\leq C\sqrt{j}\frac{1}{t}\frac{\|I_n-D_{k_t}\|_{\infty}}{\mu_{j+1}-\mu_j}
    \end{align}
    with high probability. Inserting \eqref{eq:conc:diagonal:perturbation}, the claim follows.
\end{proof}

\paragraph{Reduction to the Gaussian graph Laplacian.}  We have
\begin{align*}
    \|P_{\leq j}(\mathscr{L}_{k_t})-P_{\leq j}(\mathscr{L}_{w_t})\|_2\leq C\sqrt{j}\frac{\mu_{j+1}}{\mu_{j+1}-\mu_j}t\log^2(\tfrac{e}{t})
\end{align*}
with high probability.
\begin{proof}
For $K\in \N$ let $\tilde{\mathscr{L}}_{k_t}=\mathscr{L}_{k_t}+t^KI_n$ and $\tilde{\mathscr{L}}_{w_t}=\mathscr{L}_{w_t}+t^KI_n$. Then $\tilde{\mathscr{L}}_{k_t}$ and $\tilde{\mathscr{L}}_{w_t}$ are strictly positive and Corollary  \ref{corollary:relative:perturbation:graph:Laplacian} yields
\begin{align*}
        \|\tilde{\mathscr{L}}_{k_t}^{-1/2}(\tilde{\mathscr{L}}_{w_t}-\tilde{\mathscr{L}}_{k_t})\tilde{\mathscr{L}}_{k_t}^{-1/2}\|_\infty\leq Ct\log^2(\tfrac{e}{t})
    \end{align*}
    for all $t\in(0,c)$. In order to apply Corollary \ref{corollary:DK}, it remains to upper bound
    \begin{align*}
        \frac{\lambda_{j+1}(\tilde{\mathscr{L}}_{k_t})}{\lambda_{j+1}(\tilde{\mathscr{L}}_{k_t})-\lambda_{j}(\tilde{\mathscr{L}}_{k_t})}.
    \end{align*}
    First, by Theorem \ref{prop:heat:PCA:eigenvalue}, \eqref{eq:reduction:heat:graph:Laplacian}, \eqref{ass:Theorem:2:proof} and \eqref{eq:Weyls:law:lower:bound}, we have, with high probability,
    \begin{align*}
        \lambda_{j+1}(\tilde{\mathscr{L}}_{k_t})&=\lambda_{j+1}(\mathscr{L}_{k_t})+t^K\\
        &\leq \frac{1-e^{-\mu_{j+1}t}}{t}+C\sqrt{\frac{\log n}{n t^{d/2+2}}}+t^K\\
        &\leq \mu_{j+1}+C\sqrt{\frac{\log n}{n t^{d/2+2}}}+t^K\leq C\mu_{j+1}
    \end{align*} 
    for all $t\in(0,c)$. Second, by \eqref{eq:conc:diagonal:perturbation} and \eqref{ass:Theorem:2:proof}, we have, with high probability,
    \begin{align*}
        \lambda_{j+1}(\tilde{\mathscr{L}}_{k_t})-\lambda_{j}(\tilde{\mathscr{L}}_{k_t})
        &=\lambda_{j+1}(\mathscr{L}_{k_t})-\lambda_{j}(\mathscr{L}_{k_t})\\
        &\geq \lambda_{j+1}\Big(\frac{I-K_t}{t}\Big)-\lambda_{j}\Big(\frac{I-K_t}{t}\Big)-\frac{2}{t}\|I_n-D_{k_t}\|_{\infty}\\
        &\geq \frac{\mu_{j+1}-\mu_j}{2e}-2C\sqrt{\frac{\log n}{nt^{d/2+2}}}\\
        &\geq \frac{\mu_{j+1}-\mu_j}{4e},
    \end{align*}
    provided that $c>0$ in \eqref{ass:Theorem:2:proof} is chosen small enough (such that $c\leq 1/(8Ce)$). Applying Corollary \ref{corollary:relative:perturbation:graph:Laplacian} conclude that
    \begin{align*}
        \|P_{\leq j}(\mathscr{L}_{k_t})&-P_{\leq j}(\mathscr{L}_{w_t})\|_2
        =\|P_{\leq j}(\tilde{\mathscr{L}}_{k_t})-P_{\leq j}(\tilde{\mathscr{L}}_{w_t})\|_2\\
        &\leq C\sqrt{j}\frac{\lambda_{j+1}(\tilde{\mathscr{L}}_{k_t})}{\lambda_{j+1}(\tilde{\mathscr{L}}_{k_t})-\lambda_{j}(\tilde{\mathscr{L}}_{k_t})}\|\tilde{\mathscr{L}_{k_t}}^{-1/2}(\tilde{\mathscr{L}}_{w_t}-\tilde{\mathscr{L}}_{k_t})\tilde{\mathscr{L}_{k_t}}^{-1/2}\|_\infty\\
        &\leq C\sqrt{j}\frac{\mu_{j+1}}{\mu_{j+1}-\mu_j}t\log^2(\tfrac{e}{t}).
    \end{align*}
    for all $t\in(0,c)$.
\end{proof}

\appendix

\section{Appendix}

\subsection{Perturbation bounds}

\subsubsection{Positive self-adjoint operators}\label{Sec:perturbation:self-adjoint:operators}

Let $\Sigma$ be a self-adjoint positive operator on a separable Hilbert space $\HH$. Let $\lambda_1\geq \lambda_2\geq \dots >0$ be the sequence of positive eigenvalues and let $u_1,u_2,\dots$ be a sequence of corresponding eigenvectors. We assume that $u_1,u_2,\dots$ form an orthonormal basis of $\HH$. For $j\in\N$, let
\begin{align*}
    P_{\leq j}=\sum_{k=1}^jP_k,\qquad P_k=u_k\otimes u_k
\end{align*}
be the orthogonal projection onto the eigenspace spanned by $u_1,\dots,u_j$. Let $\hat\Sigma$ be another self-adjoint positive operator on $\HH$. Let $\hat\lambda_1\geq \hat\lambda_2\geq \dots \geq 0$ be the sequence of non-negative eigenvalues and let $\hat u_1,\hat u_2,\dots$ be a sequence of corresponding eigenvectors. We also assume that $\hat u_1,\hat u_2,\dots$ form an orthonormal basis of $\HH$. For $j\in \N$, let
\begin{align*}
    \hat P_{\leq j}=\sum_{k=1}^j\hat P_k,\qquad \hat P_k=\hat u_k\otimes \hat u_k.
\end{align*}

The following lemma states some well-known absolute perturbation bounds for eigenvalues and eigenvectors (see, e.g., \cite{MR1477662}).

\begin{lemma}[Absolute Weyl and Davis-Kahan bound]\label{lem:absolute:Weyl:DK} For all $j\in\N$, we have
\begin{align*}
    |\hat\lambda_j-\lambda_j|&\leq \|\hat\Sigma-\Sigma\|_\infty,\\
    \|\hat P_{\leq j}-P_{\leq j}\|_{2}&\leq \sqrt{32j}\frac{\|\hat\Sigma-\Sigma\|_\infty}{\lambda_j-\lambda_{j+1}}.
\end{align*}
\end{lemma}

For the case of the leading eigenvalue and eigenspace, we formulate relative perturbation bounds derived in \cite{RW17,MR4517351,W19}. Let $m\in\N$ be the multiplicity of $\lambda_1$ such that 
\begin{align*}
    \lambda_1=\dots=\lambda_m>\lambda_{m+1}.
\end{align*}
Moreover, let 
\begin{align*}
    \delta_{\leq m}(\hat\Sigma-\Sigma)=\Big\|\Big(\frac{P_{\leq m}}{\lambda_1-\lambda_{m+1}}+R_{>m}\Big)^{\frac{1}{2}}(\hat\Sigma-\Sigma)\Big(\frac{P_{\leq m}}{\lambda_1-\lambda_{m+1}}+R_{>m}\Big)^{\frac{1}{2}}\Big\|_\infty
\end{align*} 
with (outer) reduced resolvent
\begin{align*}
    R_{>m}=\sum_{k>m}\frac{P_k}{\lambda_1-\lambda_k}.
\end{align*}

\begin{lemma}\label{lem:relative:DK} We have
\begin{align*}
    \|\hat P_{\leq m}-P_{\leq m}\|_{2}&\leq 4\sqrt{2}\sqrt{m}\cdot\delta_{\leq m}(\hat\Sigma-\Sigma)
\end{align*}
and
\begin{align*}
    \|\hat P_{\leq m}-P_{\leq m}\|_{2}\leq \sqrt{2}\|P_{\leq m}(\hat\Sigma-\Sigma)R_{>m}\|_2+20\sqrt{2}m\cdot\delta_{\leq m}(\hat\Sigma-\Sigma)^2.
\end{align*}
\end{lemma}

\begin{proof}
    See Proposition 1 in \cite{JW}. For the second bound combine (2.4) in \cite{JW} with the triangle inequality.  
\end{proof}

\begin{lemma}\label{lem:relative:Weyl} Assume that $\delta_{\leq m}(\hat\Sigma-\Sigma)\leq 1/4$. Then
\begin{align*}
    |\hat\lambda_j-\lambda_1|&\leq (\lambda_1-\lambda_{m+1})\cdot\delta_{\leq m}(\hat\Sigma-\Sigma).
\end{align*}
 and
\begin{align*}
     |\hat\lambda_j-\lambda_1|\leq \|P_{\leq m}EP_{\leq m}\|_2+C(\lambda_1-\lambda_{m+1})m\cdot\delta_{\leq m}(\hat\Sigma-\Sigma)^2
\end{align*}
for all $1\leq j\leq m$, with some absolute constant $C>0$.
\end{lemma}

\begin{proof}
    Follows from proceeding as in Lemma 8 of \cite{JW} or \cite{W19}. See also Theorem 3 in \cite{MR4517351}.
\end{proof}

\begin{lemma}\label{lemma:char:subspace:distance}
    Consider $(u_1,\dots,u_j)$ and $(\hat u_1,\dots,\hat u_j)$ as linear maps from $\R^j$ to $\HH$. Then we have
    \begin{align*}
        \inf_{O\in O(j)}\|(u_1,\dots,u_j)-(\hat u_1,\dots,\hat u_j)O\|_2\leq\|\hat P_{\leq j}-P_{\leq j}\|_{2}
    \end{align*}
    with $\|\cdot\|_2$ denoting the Hilbert-Schmidt norm.
\end{lemma}

\begin{proof}
    See, e.g., Equation (2.6) and Proposition 2 in \cite{MR3161452}.
\end{proof}

\subsubsection{Symmetric matrices}

We reformulate the results of Section \ref{Sec:perturbation:self-adjoint:operators}  in the case of matrices, since our assumptions and notations for Laplacian matrices are slightly different from those for covariance operators (for instance, the eigenvalues are listed in reverse order). Let $A\in\R^{n\times n}$ be a symmetric matrix.  Let 
\begin{align*}
    \lambda_1(A)\leq \lambda_2(A)\leq  \dots\leq \lambda_n(A)
\end{align*}
be the eigenvalues of $A$ (in non-decreasing order), and let 
\begin{align*}
    u_1(A),u_2(A),\dots,u_n(A)\in \R^n
\end{align*}
be corresponding eigenvectors. For $1\leq j\leq n$, let
\begin{align*}
    P_{\leq j}(A)=\sum_{k=1}^jP_k(A),\qquad P_k(A)=u_k(A)u_k(A)^\top,
\end{align*}
be the orthogonal projection onto the eigenspace spanned by $u_1(A),\dots,u_j(A)$. For $1\leq j\leq n$ with $\lambda_{j+1}(A)>\lambda_{j}(A)$, let
\begin{align*}
    R_{>j}(A)=\sum_{k>j}\frac{P_k(A)}{\lambda_k(A)-\lambda_j(A)}.
\end{align*}
If $A,B\in\R^{n\times n}$ are symmetric, then we define
\begin{align*}
    &\delta_{\leq j}(B-A)\\
    &=\Big\|\Big(\frac{P_{\leq j}(A)}{\lambda_{j+1}(A)-\lambda_{j}(A)}+R_{>j}(A)\Big)^{\frac{1}{2}}(B-A)\Big(\frac{P_{\leq j}(A)}{\lambda_{j+1}(A)-\lambda_{j}(A)}+R_{>j}(A)\Big)^{\frac{1}{2}}\Big\|_\infty
\end{align*} 
for every $1\leq j\leq n$ such that $\lambda_{j+1}(A)>\lambda_{j}(A)$.

\begin{lemma}[Relative Davis-Kahan bound]\label{lem:relative:DK:matrix}
Let $A,B\in\R^{n\times n}$ be symmetric. Then we have
\begin{align*}
    \|\hat P_{\leq j}(A)-P_{\leq j}(A)\|_{\infty}&\leq \sqrt{32j}\cdot \delta_{\leq j}(B-A).
\end{align*}
for every $1\leq j\leq n$ such that $\lambda_{j+1}(A)>\lambda_{j}(A)$.
\end{lemma}

\begin{proof}
    See again Proposition 1 in \cite{JW}.
\end{proof}

\begin{lemma}\label{lemma:bound:delta}
    Let $A,B\in\R^{n\times n}$ be symmetric and let $1\leq j\leq n$ such that $\lambda_{j+1}(A)>\lambda_{j}(A)$. Then
    \begin{align*}
        \delta_{\leq j}(B-A)\leq \frac{\|B-A\|_\infty}{\lambda_{j+1}(A)-\lambda_{j}(A)}
    \end{align*}
    and, if $A$ is additionally positive definite, then also
    \begin{align*}
        \delta_{\leq j}(B-A)\leq \frac{\lambda_{j+1}(A)}{\lambda_{j+1}(A)-\lambda_{j}(A)}\|A^{-1/2}(B-A)A^{-1/2}\|_\infty.
    \end{align*}
\end{lemma}

\begin{proof}
    Follows from simple properties of the operator norm.
\end{proof}

\begin{corollary}\label{corollary:DK}
     Let $A,B\in\R^{n\times n}$ be symmetric and let $1\leq j\leq n$ such that $\lambda_{j+1}(A)>\lambda_{j}(A)$. Then
    \begin{align*}
    \|\hat P_{\leq j}(B)-P_{\leq j}(A)\|_{\infty}&\leq \sqrt{32j}\frac{\|B-A\|_\infty}{\lambda_{j+1}(A)-\lambda_{j}(A)}
\end{align*}
    and, if $A$ is additionally positive definite, then also
    \begin{align*}
    \|\hat P_{\leq j}(B)-P_{\leq j}(A)\|_{\infty}&\leq \sqrt{32j}\frac{\lambda_{j+1}(A)}{\lambda_{j+1}(A)-\lambda_{j}(A)}\|A^{-1/2}(B-A)A^{-1/2}\|_\infty.
\end{align*}
\end{corollary}

\begin{proof}
    Follows from inserting Lemma \ref{lemma:bound:delta} into Lemma \ref{lem:relative:DK:matrix}.
\end{proof}

\begin{lemma}\label{lemma:graph:Laplace:version:Weyl:bound}
    Let $A,B\in\R^{n\times n}$ be symmetric and let $1\leq j\leq n$. Then, 
    \begin{align*}
    |\lambda_j(B)-\lambda_j(A)|&\leq \|B-A\|_\infty
\end{align*}
    and, if $A$ is additionally strictly positive definite, then also
    \begin{align*}
    |\lambda_j(B)-\lambda_j(A)|&\leq \lambda_j(A)\|A^{-1/2}(B-A)A^{-1/2}\|_\infty.
\end{align*}
\end{lemma}

\begin{proof}
    See Theorem 2.7 in \cite{MR1689433}.
\end{proof}

\subsection{Concentration inequalities}\label{app:concentration:inequalities}

The following lemma states an operator Bernstein inequality (see, e.g., Lemma 5 in \cite{MR3629418} and also \cite{8186879}).
\begin{lemma}\label{lem:operator:Bernstein}
	Let $\xi_1,\dots,\xi_n$ be a sequence of independently and identically distributed self-adjoint
	Hilbert-Schmidt operators on a separable Hilbert space. Suppose that $\mathbb{E}\xi_1=0$ and that $\|\xi_1\|_{\operatorname{op}}\leq R$ almost surely for some constant $R>0$. Moreover, suppose that there are constants $V,D > 0$ satisfying $\|\mathbb{E}\xi_1^2\|_{\operatorname{op}}\leq V$ and $\operatorname{tr}(\mathbb{E}\xi_1^2)\leq VD$. Then, for all $u>0$,
	\begin{align*}
		\mathbb{P}\Big(\Big\|\frac{1}{n}\sum_{i=1}^n\xi_i\Big\|_{\operatorname{op}}\geq u\Big)\leq 4D\exp\Big(-\frac{nu^2}{2V+(2/3)uR}\Big).
	\end{align*}
\end{lemma}
The next lemma states an extension of Lemma A.1 in \cite{MR4652993}.
\begin{lemma}\label{lem:conc:Hilbert:norm}
    Let $(a_k)$ be a non-increasing sequence of non-negative real numbers such that $\|(a_k)\|_1=\sum_{k=1}^\infty a_k<\infty$. Let $(Z_{ik})$ be a sequence of real-valued centered random variables with $k\in\N$ and $1\leq i\leq n$. Suppose that $(Z_{ik})_{i=1}^n$ are i.i.d.~for each $k\in \N$ and that there is a constant $C_1>0$ such that 
    \begin{align*}
        \forall N\in\N,\qquad\sum_{k=1}^NZ_{1k}^2\leq C_2N\quad \text{a.s.}
    \end{align*}
    Then there is a constant $C>0$ depending only on $C_2$ such that, for every $t\geq 1$, we have
    \begin{align*}
        \sum_{k=1}^\infty a_k\Big(\frac{1}{n}\sum_{i=1}^nZ_{ik}\Big)^2\leq C\Big(\|(a_k)\|_1\cdot\frac{t}{n}+\|(ka_k)\|_1\frac{t^2}{n^2}\Big)
    \end{align*}
    with probability at least $1-e^{-t}$.
\end{lemma}

\begin{proof}
    For $p\in\N$ Minkowski's inequality yields
    \begin{align}\label{eq:Minkovski}
        \Big\|\sum_{k=1}^\infty a_k\Big(\frac{1}{n}\sum_{i=1}^nZ_{ik}\Big)^2\Big\|_{L^p}\leq \sum_{k=1}^\infty a_k\Big\|\frac{1}{n}\sum_{i=1}^nZ_{ik}\Big\|_{L^{2p}}^2
    \end{align}
    By assumption $Z_{1k}$ is centered with $|Z_{1k}|\leq \sqrt{C_2k}$. Hence, Bernstein's inequality yields
    \begin{align*}
         \PP\Big(\frac{1}{n}\sum_{i=1}^n Z_{ik}\geq \sqrt{\frac{2z\E Z_{1k}^2}{n}}+\frac{z\sqrt{C_2k}}{n}\Big)\leq e^{-z}
    \end{align*}
    for all $z>0$ and all $k\in\N$. Since the same deviation bound holds for $-n^{-1}\sum_{i=1}^nZ_{ik}$, Theorem 2.3 in \cite{MR3185193} yields 
    \begin{align*}
        \forall p\in\N,\qquad \Big\|\frac{1}{n}\sum_{i=1}^nZ_{ik}\Big\|_{L^{2p}}^2&\leq \Big(p!\Big(\frac{8\E Z_{1k}^2}{n}\Big)^p+(2p)!\Big(\frac{4\sqrt{C_2k}}{n}\Big)^{2p}\Big)^{1/p}\\
        &\leq C\Big(\frac{p}{n}\E Z_{1k}^2+\frac{p^2}{n^2}k\Big).
    \end{align*}
    Inserting this into \eqref{eq:Minkovski}, we get
    \begin{align*}
    \forall p\in\N,\qquad \Big\|\sum_{k=1}^\infty a_k\Big(\frac{1}{n}\sum_{i=1}^nZ_{ik}\Big)^2\Big\|_{L^p}\leq C\Big(\frac{p}{n}\E\sum_{k=1}^\infty a_kZ_{1k}^2+\frac{p^2}{n^2}\sum_{k=1}^\infty ka_k\Big).
    \end{align*}
    Set $B_k=\sum_{l=1}^kZ_{1l}^2$. Then $B_k\leq Ck$ for all $k\in\N$ by assumption. Hence, for all $N\in\N$, we have
    \begin{align*}
        \sum_{k=1}^N a_kZ_{1k}^2&=\sum_{k=1}^Na_k(B_k-B_{k-1})\\
        &=a_NB_N-a_1B_0+\sum_{k=1}^{N-1}(a_k-a_{k-1})B_k\\
        &\leq C\Big(a_NN+\sum_{k=1}^{N-1}(a_{k}-a_{k-1})k\Big)=C\sum_{k=1}^Na_k
    \end{align*}
    almost surely. Letting $N\rightarrow \infty$, this implies 
    \begin{align*}
        \sum_{k=1}^\infty a_kZ_{1k}^2\leq C\|(a_k)\|_1
    \end{align*}
    almost surely. We arrive at 
    \begin{align*}
    \forall p\in\N,\qquad \Big\|\sum_{k=1}^\infty a_k\Big(\frac{1}{n}\sum_{i=1}^nZ_{ik}\Big)^2\Big\|_{L^p}\leq C\Big(\frac{p}{n}\|(a_k)\|_1+\frac{p^2}{n^2}\|(ka_k)\|_1\Big).
    \end{align*}
    The claim now follows from an application of Markov's inequality with an appropriate choice of $p$ (see, e.g., Equation (3.2) in \cite{MR1857312}).
\end{proof}

\subsection{Eigenvalue estimates}

\begin{lemma}\label{lem:eigenvalue:estimate}
    Let $(\mu_j)_{j=m+1}^\infty$ be a sequence of real numbers such that $\mu_j\geq cj^{2/d}$ for all $j>m$ and some constants $d\geq 3$ and $c>0$. Then there is an absolute constant $C>0$ depending only on $c$ and $d$ such that  
    \begin{align*}
        &\sum_{j>m}\frac{e^{-\mu_jt}}{1-e^{-\mu_jt}}\leq Ct^{-d/2},\\
        &\sum_{j>m}\frac{e^{-\mu_jt}}{(1-e^{-\mu_jt})^2}\leq Ct^{-d/2},\\
        &\sum_{j>m}\frac{je^{-\mu_jt}}{(1-e^{-\mu_jt})^2}\leq Ct^{-d}
    \end{align*}
    for all $t\in (0,1]$. In the last inequality, we additionally assume that $d\geq 5$.
\end{lemma}

\begin{proof}
    Since $x\mapsto x/(1-x)$ is increasing for $x\in(0,1]$, we have
    \begin{align*}
        \sum_{j>m}\frac{e^{-\mu_jt}}{1-e^{-\mu_jt}}&\leq \sum_{j>m}\frac{e^{-cj^{2/d}t}}{1-e^{-cj^{2/d}t}}\\
        & \leq \sum_{j=m+1}^M\frac{1}{e^{cj^{2/d}t}-1}+\sum_{j=M+1}^\infty\frac{e^{-cj^{2/d}t}}{1-e^{-cj^{2/d}t}}
    \end{align*}
    for $M>m$. Choose $M=\max\{j>m:cj^{2/d}t\leq 1\}$ and $M=m$ if the latter set is empty. Then
    \begin{align*}
        \sum_{j=m+1}^M\frac{1}{e^{cj^{2/d}t}-1}&\leq \sum_{j=m+1}^M\frac{1}{cj^{2/d}t}\\
        &\leq \frac{C}{t}\int_m^M\frac{1}{x^{2/d}}\,dx\\
        &\leq Ct^{-1}M^{1-d/2}\leq Ct^{-1}(t^{-d/2})^{1-d/2}=Ct^{-d/2}
    \end{align*}
    and 
    \begin{align*}
        \sum_{j=M+1}^\infty\frac{e^{-cj^{2/d}t}}{1-e^{-cj^{2/d}t}}&\leq \frac{1}{(1-e^{-1})}\sum_{j=M+1}^\infty e^{-cx^{2/d}t},\\
        &\leq \frac{1}{1-e^{-1}}\int_M^\infty e^{-cx^{2/d}t}\, dx\\
        &=\frac{1}{1-e^{-1}}t^{-d/2}\int_{Mt^{d/2}}^\infty e^{-cx^{2/d}}\,dx\leq Ct^{-d/2}.
    \end{align*}
    The other bounds follow similarly.
\end{proof}


\bibliographystyle{plain}
\bibliography{references.bib}

\end{document}